\documentclass[letterpaper,11pt]{amsart}

\usepackage{epsfig}
\usepackage{amssymb}
\usepackage{amsfonts}
\usepackage{amsmath}
\usepackage{graphicx}
\usepackage{mathrsfs}
\usepackage{stmaryrd}
\usepackage{amsthm}
\usepackage{mathtools}
\usepackage[all]{xy}
\usepackage[top=1in, bottom=1in, left=1in, right=1in]{geometry}
\usepackage{color}

\newtheorem{theorem}{Theorem}
\newtheorem{lemma}{Lemma}

\newtheorem{corollary}{Corollary}

\newtheorem{definition}{Definition}
\newtheorem{remark}{Remark}

\numberwithin{equation}{section}
\numberwithin{lemma}{section}
\numberwithin{proposition}{section}
\numberwithin{corollary}{section}
\numberwithin{remark}{section}
\numberwithin{theorem}{section}
\numberwithin{definition}{section}

\UseRawInputEncoding

\title{A note on the Compactness of Poincar\'{e}-Einstein manifolds}

\author{Fang Wang}
\address{School of Mathematical Sciences, Shanghai Jiao Tong University, 800 Dongchuan Rd, Shanghai, 200240}
\email{fangwang1984@sjtu.edu.cn}
\author{Huihuang Zhou}
\address{School of Mathematical Sciences, Shanghai Jiao Tong University, 800 Dongchuan Rd, Shanghai, 200240}
\email{zhouhuihuang@sjtu.edu.cn}
\date{}
\thanks{Wang and Zhou's research are supported in part by National Natural Science Foundation of China No. 11871331 
and Shanghai Science and Technology Innovation Action Plan No. 20JC1413100.}

\begin{document}

\begin{abstract}
For a conformally compact  Poincar\'{e}-Einstein manifold $(X,g_+)$, we consider two types of compactifications for it. One is $\bar{g}=\rho^2g_+$, where $\rho$ is a fixed smooth defining function; the other is  the adapted (including Fefferman-Graham) compactification $\bar{g}_s=\rho^2_sg_+$ with a continuous parameter $s>\frac{n}{2}$. 
In this paper, we mainly prove that  for a set of  conformally compact Poincar\'{e}-Einstein manifolds $\{(X, g_{+}^{(i)})\}$ with conformal infinity of positive Yamabe type,  $\{\bar{g}^{(i)}\}$ is compact in $C^{k,\alpha}(\overline{X})$ topology if and only if $\{\bar{g}_s^{(i)}\}$ is compact in some $C^{l,\beta}(\overline{X})$ topology, provided that $\bar{g}^{(i)}|_{TM}=\bar{g}_s^{(i)}|_{TM}=\hat{g}^{(i)}$  and $\hat{g}^{(i)}$ has positive scalar curvature for each $i$. See Theorem \ref{thm.main} and Corollary \ref{cor.main} for the exact relation of $(k,\alpha)$ and $(l,\beta)$.


\end{abstract}

\maketitle

\section{Introduction}


Suppose $\overline{X}^{n+1}$ is a $(n+1)$-dimensional smooth compact manifold with boundary. Denote $X$ the interior of $\overline{X}$ and $M=\partial X$ the boundary. 
Let $\rho$ be a smooth boundary defining function, i.e., 
$$
0\leq \rho\in C^{\infty}(\overline{X}), \quad \rho>0\ \textrm{in $X$},\quad \rho=0\ \textrm{on $M$},  \quad d\rho|_{M}\neq 0. 
$$
We call  $(X,g_+)$ a $C^{k,\alpha}$ \textit{conformally compact Poincar\'e-Einstein manifold} with \textit{conformal infinity} $(M,[\hat{g}])$, if $g_+$ is a  complete metric in $X$ such that $\bar{g}=\rho^2g_+$ can be $C^{k,\alpha}$ extended to $\overline{X}$ and satisfies
$$
\begin{cases}
\mathrm{Ric}_{g_+}=-ng_+ &\ \mathrm{in}\ X,
\\
\bar{g}|_{TM}\in [\hat{g}] &\ \mathrm{on}\ M.
\end{cases}
$$
In this paper, we always require that $n\geq 3$, $k\geq 3, 0<\alpha<1$ and $\rho$ is fixed.

Given $(M, [\hat{g} ])$, the existence and uniqueness of $(X, g_{+})$, such that $(X, g_{+})$ is Poincar\'{e}-Einstein with conformal infinity $(M, [\hat{g}])$, are difficult problems in general. See \cite{An1, An2,An3, GL1, ST1, LSQ1} and many other works for related results.  
The compactness problem was developed recently in a series of papers by Chang-Ge-Qing. See \cite{CGQ1, CGQ2}. 

The compactness problem is formulated as follows: Given a set of $C^{k,\alpha}$ conformally compact Poincar\'{e}-Einstein metrics $\{(X, g_{+}^{(i)})\}$ with conformal infinity $\{(M, [\hat{g}^{(i)} ])\}$, if the boundary representatives  $\{\hat{g}^{(i)}\}$ is compact in $C^{k,\alpha}(M)$   topology, whether $\displaystyle \{\bar{g}^{(i)}=(\rho^{(i)})^2 g_+^{(i)}\}$  is also compact in $C^{k,\alpha}(\overline{X})$   topology  by a suitable choice of $\{\rho^{(i)}\}$ such that $\bar{g}^{(i)}|_{TM}=\hat{g}^{(i)}$. 
In \cite{CGQ1, CGQ2}, the authors proved that  it is true under certain geometric assumptions for the case $M=\mathbb{S}^3$. In particular, they chose $\{\bar{g}^{(i)}\}$ to be  the Fefferman-Graham compactification of $\{g_+^{(i)}\}$ when deriving the compactness, because for $n=3$, this imposes an extra curvature vanishing condition for $\{\bar{g}^{(i)}\}$.

In this paper, we mainly want to show that once we get the compactness for one type of compactification, then it may also hold for some other type. 
We consider a family of compactifications defined by a real parameter $s$ 
for a given $C^{k,\alpha}$ conformally compact Poincar\'{e}-Einstein manifold $(X,g_+)$ with conformal infinity $(M, [\hat{g}])$: 
\begin{itemize}
\item[(i)] 	For $s>\frac{n}{2}, s\neq n$, if $\mathrm{Spec}(-\Delta_+)>s(n-s)$, let $v_s$ satisfy the following equation:
\begin{equation*}
-\Delta_+ v_s- s(n-s)v_s=0, \quad 
v_s=\rho^{n-s}\left(1+O(\rho^{\epsilon})\right)\ \textrm{for some $\epsilon>0$}.
\end{equation*}
Then the \textit{adapted compactification} of $g_+$ is defined by
$$
\bar{g}_s=\rho_s^2 g_+, 
\quad\mathrm{where}\quad 
\rho_s=v_s^{\frac{1}{n-s}}.
$$

\item[(ii)]  For $s=n$, let $w$ satisfy the following equation
\begin{equation*}
 	-\Delta_+ w=n, \quad w=\log\rho+O(\rho^{\epsilon})\ \textrm{for some $\epsilon>0$}.
 \end{equation*}
 Then the \textit{Fefferman-Graham compactification} of $g_+$ is defined by
 $$
 \bar{g}_F=\rho_F^2g_+, 
 \quad\mathrm{where}\quad 
 \rho_F=e^{w}.
 $$
\end{itemize}
The spectrum condition required in (i) can be satisfied by a nonnegativity assumption on the Yamabe constant of the conformal infinity. 
In \cite{Le1}, Lee  showed that if the boundary Yamabe constant $\mathcal{Y}(M, [\hat{g}]) \geq 0$, then $\mathrm{Spec}(-\Delta_+)\geq \frac{n^2}{4}$. In this case, we can view the Fefferman-Graham compactification as  a limit of the adapted compactificaions: 
$$
\lim_{s\rightarrow n}\bar{g}_s=\bar{g}_F.
$$
Hence by taking $g_n=g_F$,  $\{\bar{g}_s:s>\frac{n}{2}\}$ forms a continuous family of compactifications.  See Section \ref{sec.4} for the details. 

The adapted  and Fefferman-Graham compactifications were first introduced in \cite{CC1} and \cite{FG2} respectively, which satisfy certain geometric curvature conditions and have certain advantage to study the fractional GJMS operators on the boundary, as well as the correspondence of interior Einstein metric with boundary conformal geometry. 
For example,  while $s=\frac{n}{2}+N-\frac{1}{2}$ for some positive integer $N$, then the  $2N$-th order Q-curvature of $\bar{g}_s$ vanishes, i.e. $Q_{2N}^{\bar{g}_s}=0$.

\vspace{0.1in}
The main theorem of this paper is the following.

\begin{theorem}\label{thm.main}
Suppose $\{(X,g^{(i)}_+)\}$ is a set of conformally compact Poincar\'{e}-Einstein manifolds with conformal infinity $\{(M,[\hat{g}^{(i)}])\}$ satisfying $\mathcal{Y}(M,[\hat{g}^{(i)}])\geq 0$. 
Let $\rho$ be a fixed smooth defining function and 
$$
\bar{g}^{(i)}=\rho^2g^{(i)}_+, \quad \bar{g}^{(i)}|_{TM}=\hat{g}^{(i)}. 
$$
For $s>\frac{n}{2}$, let $\{\bar{g}^{(i)}_s\}$ be the adapted (including Fefferman-Graham) compactification of $g^{(i)}_+$ with the fixed boundary representatives
$$
\bar{g}_s^{(i)}|_{TM}=\bar{g}^{(i)}|_{TM}=\hat{g}^{(i)}.
$$
\begin{itemize}
\item[(a)] 
If $\{\bar{g}^{(i)}\}$ is compact in $C^{k,\alpha}(\overline{X})$ $(k\geq 3, 0<\alpha<1)$   topology, 
then $\{\bar{g}_s^{(i)}\}$ is compact in $C^{l,\beta}(\overline{X})$ $(l\geq 0,0<\beta<1)$   topology, 
where
\begin{itemize}
\item[(a.1)] $l=k,0<\beta<\alpha$ if $k+\alpha<2s-n$;
\item[(a.2)] $l=k,0<\beta<\alpha$ if $k+\alpha\geq 2s-n$ and $2s-n=2N-1$ for some positive integer $N$;
\item[(a.3)] $l+\beta<2s-n$ if $k+\alpha\geq 2s-n$ and $2s-n$ is not a positive odd integer. 
\end{itemize}

\item[(b)] 
If $\{\bar{g}_s^{(i)}\}$ is compact in $C^{k,\alpha}(\overline{X})$ $(k\geq 3, 0<\alpha<1)$    topology for some  $s>\frac{n}{2}+1$ or $s=\frac{n+1}{2}$,  and  the boundary scalar curvature $\hat{R}^{(i)}>0$ for all $i$, 
then $\{\bar{g}^{(i)}\}$ is compact in $C^{k,\beta}(\overline{X})$ $(0<\beta<\alpha)$   topology.
\end{itemize}
\end{theorem}

The proof of statement (a) is mainly based on the uniform  H\"{o}lder norm estimates  for the adapted  (including Fefferman-Graham)  boundary defining function $\rho_s$ and the ratio  $\rho_s/\rho$, where $\rho$ is the fixed boundary defining function (background coordinate). 
Here the main difficulty is that the equations  which determine $\rho_s$ and  $\rho_s/\rho$ are degenerate near the boundary and hence  the classical Schauder estimates fail. 
We overcome this difficulty by combining asymptotic analysis and delicate extension theorem (Theorem \ref{thm.ext})  with improved Schauder estimates on weighted  H\"{o}lder spaces (Definition \ref{definition}) for uniformly degenerate operators given by Lee in \cite{Le2}. These improved  Schauder estimates  are only valid for the weight in a certain indicial gap determined by the  operators. Hence for $k+\alpha\geq 2s-n$,  it implies a loss of regularity in general. However, while $2s-n=2N-1$ for some positive integer $N$, we can recover the full regularity by using the Q-curvature equation $Q_{2N}^{\bar{g}_s}=0$. Otherwise, we  can view this loss of regularity as the result of the fractional order term $\rho^{2s-n}$ in the asymptotical expansion of $\rho_s$ at boundary. 

The proof of statement (b) is mainly by exploring the geometric positivities for $\bar{g}_s$ under the assumption $\hat{R}>0$, which gives $L^{\infty}$ estimates for $\rho_s$ and $|d\rho_s|_{\bar{g}_s}$ directly. In this case, the equations that determine $\rho_s$ are uniformly elliptic with background metric $\bar{g}_s$, and hence higher order norm estimates follow from the classical Schauder estimates. 
Moreover, 	$\hat{R}>0$ also implies that for $\gamma\in(0,1)$ the fractional Q-curvature $Q^{\hat{g}}_{2\gamma}$ is  positive on the boundary  by \cite{GQ1}. Hence if $s \in (\frac{n}{2}, \frac{n}{2}+1)$,  $\{\bar{g}_s\}$  can not have $C^{k,\alpha}(\overline{X})$ regularity except for $s=\frac{n+1}{2}$. This is why we only consider $s>\frac{n}{2}+1$ or $s=\frac{n+1}{2}$.

\vspace{0.1in}
Theorem \ref{thm.main} implies the following corollary directly. 

\begin{corollary}\label{cor.main}
Suppose $\{(X,g^{(i)}_+)\}$ is a family of conformally compact Poincar\'{e}-Einstein manifold with conformal infinity $\{(M,[\hat{g}^{(i)}])\}$ satisifying $\mathcal{Y}(M,[\hat{g}^{(i)}])\geq 0$. 
Let $\rho$ be a fixed smooth defining function and  $ \rho^2g^{(i)}_+|_{TM}=\hat{g}^{(i)}. $
Let $N,K$ be two positive integers and $s=\frac{n}{2}+N-\frac{1}{2}, t=\frac{n}{2}+K-\frac{1}{2}$. 
Assume that
$$
\bar{g}_s^{(i)}|_{TM}=\bar{g}^{(i)}_{t}|_{TM}=\hat{g}^{(i)}
$$
satisfying $\hat{R}^{(i)}>0$ for all $i$. 
If $\{\bar{g}_s^{(i)}\}$ is compact in $C^{k,\alpha}(\overline{X})$ $(k\geq 3, 0<\alpha<1)$   topology, then $\{\bar{g}_{t}^{(i)}\}$ is also compact in $C^{k,\beta}(\overline{X})$ $(0<\beta<\alpha)$   topology. 
\end{corollary}

The outline of the paper is as follows.
In Section \ref{sec.2}, we review some basic properties for Poincar\'{e}-Einstein manifolds and its conformal compactification.
In Section \ref{sec.3}, we introduce the improved Schauder estimates and invertibility theorem for $\Delta_++s(n-s)$ given by Lee in \cite{Le2}.
In Section \ref{sec.4}, we prove the global regularity theorem for $\bar{g}_s$. 
In Section \ref{sec.5}, we investigate the geometric positivity results for $\bar{g}_s$ under the assumption $\hat{R}>0$ for $\hat{g}=\bar{g}_s|_{TM}$. 
In Section \ref{sec.6}, we prove Theorem \ref{thm.main}. 
In Section \ref{sec.app}, we provide an extension theorem, which is used in the proof of global regularity of $\bar{g}_s$.

\vspace{0.1in}
\textbf{Achknowlegement:} The first author wants to thank Professor Sun-Yung Alice Chang for proposing the question about the compactness in different type of compactifications for Poincar\'{e}-Einstein manifolds.

\vspace{0.2in}
\section{Poincar\'{e}-Einstein Manifold and Its Conformal Compactification}\label{sec.2}

Suppose $(X^{n+1}, g_+)$ ($n\geq 3$) is a $C^{k,\alpha}$ conformally compact Poincar\'{e}-Einstein manifold with conformal infinity $(M, [\hat{g}])$ with $k\geq 3$, $0<\alpha<1$. Let $\rho$ be a fixed smooth boundary defining function. Denote 
$$
\bar{g}=\rho^2g_+
\quad\textrm{and}\quad
\bar{g}|_{M}=\hat{g}. 
$$

For any $l\geq 0$ and $0\leq \beta<1$, denote $C^{l,\beta}(X)$ and $C^{l,\beta}(\overline{X})$ the classical H\"{o}lder spaces consists of functions with local and uniform $C^{l,\beta}$ H\"{o}lder continuity. 

\subsection{Conformal Transformation}
Let $\bar{R}, \bar{R}_{ij}$ be the scalar curvature and Ricci curvature tensor of $\bar{g}$. Let $\bar{J}=\frac{1}{2n}\bar{R}$ and $\bar{E}_{ij}$ be the trace free part of  $\bar{R}_{ij}$. Let $H$ be the boundary mean curvature for $(\overline{X}, \bar{g})$. 
Notice that for $k\geq 3$, a $C^{k,\alpha}$ compactification of Poincar\'{e}-Einstein manifold, $(\overline{X}, \bar{g})$ always has an umbilical boundary. Hence the second fundamental form of the boundary is 
$$
\Pi=\frac{1}{n} {H}\hat{g}.
$$
Let $\Delta_+$ and $\bar{\Delta}$ be the Beltrami-Laplace operators for $g_+$ and $\bar{g}$ respectively. The conformal transformation implies the following relation between the two operators:
\begin{equation}\label{eq.Delta}
\Delta_+=\rho^2\bar{\Delta}-(n-1)\rho\langle d\rho, d \cdot\rangle_{\bar{g}};
\end{equation}
and it also gives the scalar curvature and Ricci curvature of $\bar{g}$ in terms of $\rho$: 
\begin{equation}\label{eq.R}
\bar{J}= -\rho^{-1}\bar{\Delta}\rho -\frac{n+1}{2} \rho^{-2}\left(1-|d\rho|^2_{\bar{g}}\right),
\end{equation}
\begin{equation}\label{eq.E}
	\bar{E}= -(n-1) \rho^{-1}\left(\bar{\nabla}^2\rho-\frac{1}{n+1}(\bar{\Delta}\rho) \bar{g}\right). 
\end{equation}

Notice that by assuming $\bar{g}\in C^{k,\alpha}(\overline{X})$, we have $|d\rho|^2\in C^{k,\alpha}(\overline{X})$ and  (\ref{eq.R}) implies that
$|d\rho|_{\bar{g}}^2=1$ on $M$. This means the sectional curvature of $g_+$ approaches to $-1$ at the boundary by \cite{Ma1}. 
We also define the \textit{T-curvature} of $\bar{g}$ as follows: 
\begin{equation}\label{eq.T}
\bar{T}=\rho^{-1}  \left(1-|d\rho|^2_{\bar{g}}\right).
\end{equation}
Then 
$\bar{T}\in C^{k,\alpha}(X)\cap C^{k-1,\alpha}(\overline{X}). $
By (\ref{eq.R}) again, 
$$
\rho\bar{J}=-\bar{\Delta}\rho-\frac{n+1}{2}T. 
$$
Therefore, 
$\bar{J},\bar{E} \in C^{k-1,\alpha}(X)\cap C^{k-2,\alpha}(\overline{X})$.

\subsection{GJMS Operators}
We recall the GJMS operators for $(X,g_+)$ and $(\overline{X}, \bar{g})$ here. Since  $g_+$ and $\bar{g}$ are (conformal) Einstein metrics, the GJMS operators of all orders are well defined for them. Let $P^+_{2N}$ and $\bar{P}_{2N}$  be the GJMS operators of order $2N$ w.r.t. $g_+$ and $\bar{g}$ respectively.
Here we only consider the case $2\leq 2N\leq k$ because of the finite regularity. Then
$$
P^+_{2N}=\Pi_{j=1}^N
\left(-\Delta_+-\frac{(n+2N-4j+3)(n-2N+4j-3)}{4}
\right),
$$
$$
\bar{P}_{2N}=\rho^{-\frac{n+1}{2}-N} P^+_{2N} \rho^{\frac{n+1}{2}-N}
=\left(-\bar{\Delta}\right)^N + L. O. T. ,
$$
where $L.O.T.$ denotes a differential operator of order $\leq 2N-2$. 
The \textit{Q-curvature} of the corresponding order is given by 
$$
\bar{Q}_{2N}=\frac{2}{n+1-2N} \bar{P}_{2N}(1). 
$$
If $2N=n+1$, then the formula is understood as a formal "limit" of $n+1\rightarrow 2N$.  See \cite{GJMS, FG1, Go1, CC1} for more details. In particular, for $N=1$, 
$$
\bar{Q}_2=\frac{1}{2n}\bar{R}=\bar{J}. 
$$
For $N=2$, 
$$
\bar{Q}_4=-\bar{\Delta}\bar{J}+\frac{n+1}{2}\bar{J}^2-2|\bar{A}|^2, 
$$
where $\bar{A}$ is the Schouten tensor of $\bar{g}$, i.e.
$$
\bar{A}=\frac{1}{n-1}\left(\bar{R}_{ij}-\bar{J}\bar{g}_{ij}\right). 
$$
Recursive formulae given in \cite{Ju1} show that for all $N \geq 1$, 
$$
\bar{Q}_{2N}=(-\Delta)^{N-1}\bar{J} + L.O.T., 
$$
where $L.O.T$ is a polynomial of derivatives of metric $\bar{g}$ of order $\leq 2N-2$ with coefficients given by $\bar{g}$ and $\bar{g}^{-1}$ in local coordinates.

\section{Analysis on Poincar\'{e}-Einstein manifold}\label{sec.3}
In this section, we recall some analysis on a $C^{k,\alpha}$ conformally compact Poincar\'{e}-Einstein manifold $(X^{n+1}, g_+)$. Since our background metric $g_+$ has only finite regularity, we will mainly follow from \cite{Le2}. Let $\rho$ be the fixed smooth boundary defining function and  
 $$\bar{g}=\rho^2g_+, \quad\hat{g}=\bar{g}|_{TM}.$$
Denote $X_{\epsilon}=\{0<\rho<\epsilon\}$ for $\epsilon>0$ small. Then $X\backslash X_{\epsilon}$ is a compact subset of $X$.

\subsection{M\"{o}bius Coordinates}
First, we can cover $M$ by finitely many smooth coordinate charts $(\Omega, \Theta)$, called the \textit{background coordinates}, where $\Theta=(\rho,\theta)=(\rho, \theta^1, \cdots,  \theta^n)$. 
After fixing a choice of such background coordinates $(\Omega, \Theta)$, they form a finite covering of a neighbourhood $W$ of $M$ in $\overline{X}$. Then there is a positive number $c>0$ such that $X_c\subset W$ and every point $p\in X_c$ is contained in a background coordinate chart  containing a set  of the following form
$$
\{(\rho,\theta): |\theta-\theta_p|<c, 0\leq \rho <c\}.
$$

Let $\mathbb{H}$ be the half plane model of the hyperbolic space with coordinates $(x,y)=(x,y^1, \cdots y^n)$. Let $g_{\mathbb{H}}$ be the hyperbolic metric $g_{\mathbb{H}} =x^{-2}(dx^2+dy^2)$. Let $B_r$ be the geodesic ball in $\mathbb{H}$ with radius $r$ and center $(1,0,\cdots, 0)$. 
If $p_0\in X_{c/8}$, then choose the background chart containing $p_0$ and define  a \textit{M\"{o}bius chart} centered at $p_0$ by
$$
\Phi_{p_0}: B_2\longrightarrow X, \quad (\rho,\theta)=\Phi_{p_0}(x,y)=(\theta_0+\rho_0x, \rho_0y), 
$$
where $(\rho_0,\theta_0)$ is the background coordinates of $p_0$. It is easy to check that for all $0<r\leq 2$, $V_r(p_0)=\Phi_{p_0}(B_r)\subset\Phi_{p_0}(\overline{B_2})\subset X_c$. We can also choose finite smooth coordinate charts $\Phi_i: B_2\longrightarrow X$ such that the sets $\Phi_i(B_1)$ cover a neighborhood of $X\backslash X_{c/8}$. For consistency, we will also call these "M\"{o}bius Charts". Without loss of generality, we can require that $\Phi$ extends smoothly to a neighborhood of $B_2$.  We recall the following two properties for M\"{o}bius charts given in \cite{Le2}. 

\begin{lemma}\label{lem.MC1}
There exists a constant $C>0$ such that if $\Phi_{p_0}: B_2\rightarrow X$ is any M\"{o}bius chart, then
$$
\|\Phi^*_{p_0} g_+ - g_{\mathbb{H}}\|_{C^{k,\alpha}(\overline{B_2})}\leq C, 
\quad
\sup_{B_2}\left(|(\Phi^*_{p_0}g_+)^{-1}g_{\mathbb{H}}|+|(\Phi^*_{p_0}g_+)g_{\mathbb{H}}^{-1}|\right) \leq C, 
$$	
where $\|\cdot\|_{C^{k,\alpha}(\overline{B_2})}$ is the classical Euclidean H\"{o}lder norm.
\end{lemma}
\begin{lemma}\label{lem.MC2}
There exists a countable collection of points $p_i\subset X$ and corresponding M\"{o}bius chart: $\Phi_i=\Phi_{p_i}: B_2\longrightarrow V_2(p_i)$ such that $\{V_1(p_i)\}$ cover $X$ and  $\{V_2(p_i)\}$  are uniformly locally finite.
\end{lemma}
In this paper, $(\overline{X},\rho)$ are always fixed and hence  the choice of background coordinates and M\"{o}bius charts can also be fixed, while dealing with different metrics. 
\begin{remark}\label{rem.1}
If $\{g_+^{(i)} \}$ is a set of $C^{k,\alpha}$ conformally compact Poincar\'{e}-Einstein metric such that $\{\bar{g}^{(i)}=\rho^2g_+^{(i)} \}$ is compact in the $C^{k,\alpha}(\overline{X})$  topology, then Lemma \ref{lem.MC1} holds for all $i$ with a uniform constant $C>0$. 
\end{remark}

\subsection{Function Spaces}
For $l\geq 0$ and $0\leq \beta<1$, recall that $C^{l,\beta}(X)$ and $C^{l,\beta}(\overline{X})$ are the classical H\"{o}lder spaces consisting of functions with local and uniform $C^{l,\beta}$ H\"{o}lder continuity respectively. Here we introduce two types of function spaces dealing with the asymptotical behavior of functions at the boundary.

For  $0\leq \delta \leq l+\beta$, define a subspace  $C^{l,\beta}_{(\delta)}(\overline{X}) \subset C^{l,\beta}(\overline{X}) $  by
$$
C^{l,\beta}_{(\delta)}(\overline{X})=\{u\in C^{l,\beta}(\overline{X}): u=O(\rho^{\delta})\}. 
$$
It is obviously that $C^{l,\beta}_{(0)}(\overline{X})=C^{l,\beta}(\overline{X})$.  The following Lemma is proved in  \cite[Lemma 3.1]{Le2}.

\begin{lemma}[Lee]\label{lem.FS3}
Let $0\leq \beta<1$ and $0\leq \delta \leq l+\beta$.
\begin{itemize}
\item[(1)] 	
$ \displaystyle C^{l,\beta}_{(\delta)}(\overline{X})
=\{u\in C^{l,\beta}(\overline{X}): \partial_{\rho}^ju|_{M}=0, \ \forall \ 0\leq j<\delta \}.  $
\item[(2)] 
$\displaystyle C^{l,\beta}_{(\delta)}(\overline{X})$ 
is a closed subspace of $C^{l,\beta}(\overline{X})$. 

\item[(3)]
If $j$ is a positive integer and $j-1<\delta\leq j\leq l$, then 
$\displaystyle C^{l,\beta}_{(\delta)}(\overline{X})=C^{l,\beta}_{(j)}(\overline{X})$. 

\item[(4)]
If $l<\delta \leq l+\beta$, then 
$\displaystyle  C^{l,\beta}_{(\delta)}(\overline{X})=C^{l,\beta}_{(l+\beta)}(\overline{X})$. 

\end{itemize}
\end{lemma}

\begin{definition}\label{definition}

For $\delta\in\mathbb{R}$, define the \textit{weighted H\"{o}lder space} $\Lambda^{l,\beta}_{\delta}(X)$ as follows 
$$
\Lambda^{l,\beta}_{\delta}(X)=\{u\in C^{l,\beta}(X): \|u\|_{\Lambda^{l,\beta}_{\delta}(X)}<\infty\};
$$
the norm $\|\cdot \|_{\Lambda^{l,\beta}_{\delta}(X)}$ is given by
$$
	\|u\|_{\Lambda^{l,\beta}_{\delta}(X)} :=\|\rho^{-\delta}u\|_{\Lambda^{l,\beta}_0(X)}
	\quad\textrm{and}\quad	
	\|u\|_{\Lambda^{l,\beta}_0(X)} :=\sup_{\Phi} \|\Phi^* u\|_{C^{l,\beta}(\overline{B_2})},
$$
where the supremum is over all M\"{o}bius charts. 
\end{definition}
Notice that $\Lambda^{l,\beta}_0(X) \subsetneq C^{l,\beta}(X)$. 
The main relation between $C_{(\delta)}^{l,\beta}(\overline{X})$ and  $\Lambda^{l,\beta}_{\delta}(X)$ can be seen from the following lemma, where the proof is given in \cite[Lemma 3.7]{Le2}. 

\begin{lemma}[Lee]\label{lem.FS2}
Let $0<\beta<1$ and $0\leq \delta \leq l+\beta$. The following inclusions are continuous
\begin{itemize}
\item[(1)]  $ C^{l,\beta}_{(\delta)}(\overline{X}) \hookrightarrow \Lambda^{l,\beta}_{\delta}(X)$. 
\item[(2)]  $\Lambda^{l,\beta}_{l+\beta}(X) \hookrightarrow C^{l,\beta}(\overline{X})$. 
\end{itemize}
\end{lemma}
\begin{lemma}\label{lem.FS1}
For $0\leq l+\beta<l'+\beta'$ and $\delta<\delta'$, the inclusion $\Lambda^{l',\beta'}_{\delta'}(X)\hookrightarrow \Lambda^{l,\beta}_{\delta}(X)$ is compact. 	
\end{lemma}

\subsection{Schauder Estimates}
In this paper, we mainly deal  with the following elliptic operator: 
\begin{equation}\label{eq.ellP}
	P_0=\Delta_+ + s(n-s)	
\end{equation}
for some given $s>\frac{n}{2}$. 
By Lemma 4.8 of \cite{Le2}, we have the following.

\begin{lemma}[Schauder Estimates]\label{lem.SE}
Suppose $\delta\in\mathbb{R}$, $0\leq  \beta<1$ and $2\leq l+\beta\leq k+1+\alpha$. If $u\in \Lambda_{\delta}^{0,0}(X)$ such that $P_0u\in \Lambda_{\delta}^{l-2,\beta}(X)$, then there exists a constant $C>0$ such that
$$
	\|u\|_{\Lambda^{l,\beta}_{\delta}(X)} \leq C\left(\|P_0u\|_{\Lambda^{l-2,\beta}_{\delta}(X)}+\|u\|_{\Lambda^{0,0}_{\delta}(X)}\right). 
$$
\end{lemma}
\begin{proof}
Let $\{\Phi_i\}$ be the uniform locally finite covering of $X$ by M\"{o}bius charts as in Lemma \ref{lem.MC2}. Then the estimate we want reduces to the interior Schauder estimates for $L_i=\Phi_i^*\circ P\circ (\Phi_i^{-1})^*$ in $B_2$. By assumption, $L_i$ has uniform $C^{k-1,\alpha}(\overline{B_2})$ coefficients. Hence
$$
\|u\|_{C^{l,\beta}(\overline{B_1})} \leq C(\|L_iu\|_{C^{l-2,\beta}(\overline{B_2})}+\|u\|_{C^{0,0}(\overline{B_2})}).
$$
Here $C$ only depends on the $C^{l-2,\beta}(\overline{B_2})$ norm of the coefficients of $L_i$. 
\end{proof}

\begin{remark}\label{rem.2}
If $\{g_+^{(i)}\}$ is a set of $C^{k,\alpha}$ conformally compact Poincar\'{e}-Einstein metric such that $\{\bar{g}^{(i)}=\rho^2g_+^{(i)}\}$ is compact in the $C^{k,\alpha}(\overline{X})$  topology, then Lemma \ref{lem.SE} holds for all $P_0^{(i)}$ with a  constant $C>0$ independent of $i$.
\end{remark}

\subsection{Fredholm and Invertibility}
First by \cite[Theorem C]{Le2} (which was also proved in the smooth setting in \cite{MM1, Ma1}), $P_0$ defined in (\ref{eq.ellP})  has the following properties. 

\begin{lemma}\label{lem.FH1}
For $0\leq \beta<1$, $2\leq l+\beta\leq k+\alpha$ and $n-s<\delta<s$, the map
\begin{equation}\label{eq.FH}
P_0: \Lambda_{\delta}^{l,\beta}(X) \longrightarrow \Lambda_{\delta}^{l-2,\beta}(X)
\end{equation}
 is Fredholm with kernel equal to the $L^2$-kernel of $P_0$. 
 If further $\mathrm{Spec}(-\Delta_+)>s(n-s)$, then (\ref{eq.FH}) is invertible. 
\end{lemma}
\begin{remark}\label{rmk.3}
If $\{g_+^{(i)}\}$ is a set of $C^{k,\alpha}$ conformally compact Poincar\'{e}-Einstein metric such that $\{\bar{g}^{(i)}=\rho^2g_+^{(i)}\}$ is compact in the $C^{k,\alpha}(\overline{X})$   topology and $\mathcal{Y}(M,[\hat{g}^{(i)}])\geq 0$ for all $i$, 
then for $0\leq \beta<1$, $2< l+\beta\leq k+\alpha$ and $n-s<\delta<s$, the map
\begin{equation*}
P^{(i)}_0: \Lambda_{\delta}^{l,\beta}(X) \longrightarrow \Lambda_{\delta}^{l-2,\beta}(X)
\end{equation*}
 is invertible for each $i$ and there is a constant $C>0$  independent of $i$ such that
 $$
 \|u\|_{\Lambda_{\delta}^{l,\beta}(X)} \leq C\|P_0^{(i)}u\|_{\Lambda_{\delta}^{l-2,\beta}(X)}.  
 $$
\end{remark}
\begin{proof}
By Lemma \ref{lem.SE}, \ref{lem.FH1} and Remark \ref{rem.2}, we only need to show that
if $\|P_0^{(i)}u\|_{\Lambda_{\delta}^{l-2,\beta}(X)}=1$ then there is a uniform constant $C$ such that $\|u\|_{\Lambda_{\delta}^{0,0}(X)}\leq C$. Otherwise, there exists a sequence 
$\{u_k\}_{k\geq 1}$ such that 
$$
\|P_0^{(k)}u_k\|_{\Lambda_{\delta}^{l-2,\beta}(X)}=1, \quad \|u_k\|_{\Lambda_{\delta}^{0,0}(X)}\geq k. 
$$
Take $v_k=\frac{u_k}{ \|u_k\|_{\Lambda_{\delta}^{0,0}(X)}}$. Then
$$
\begin{aligned}
&\|P_0^{(k)}v_k\|_{\Lambda_{\delta}^{l-2,\beta}(X)}\leq \frac{1}{k}, \quad \|v_k\|_{\Lambda_{\delta}^{0,0}(X)}=1, \quad
\Longrightarrow\quad
\|v_k\|_{\Lambda_{\delta}^{l,\beta}(X)}\leq  C'
\end{aligned}
$$
for some $C'>0$ independent of $i$. Take $2\leq l'+\beta'<l+\beta$ and $n-s<\delta'<\delta<s$. Then by passing to a subsequence, we can assume 
$$
v_k \longrightarrow v_{\infty} \quad \textrm{in $\Lambda_{\delta'}^{l',\beta'}(X)$}
\quad\textrm{and}\quad
\bar{g}^{k} \longrightarrow \bar{g}^{(\infty)}  \quad \textrm{in $C^{k,\alpha}(\overline{X})$}.
$$
Hence $P_0^{(k)}v_k$ converges to $P_{0}^{(\infty)}v_{\infty}$ in $\Lambda_{\delta'}^{l'-2,\beta'}(X)$ and by assumption $P_{0}^{(\infty)}v_{\infty}=0$. 
However, $P_{0}^{(\infty)}: \Lambda_{\delta'}^{l',\beta'}(X) \longrightarrow \Lambda_{\delta'}^{l'-2,\beta'}(X)$ is invertible. Hence $v_{\infty}=0$, which contradicts with $\|v_k\|_{\Lambda_{\delta}^{0,0}(X)}=1$ for all $k$. 
 We finish the proof.
\end{proof}

For $\lambda \in \mathbb{R}$, we can also define
\begin{equation}\label{eq.ellP2}
P_{\lambda}=\rho^{-\lambda}(\Delta_++s(n-s))\rho^{\lambda}. 
\end{equation}
Then a shifting of the weight shows the following.
\begin{lemma}\label{lem.FH2}
For $0\leq \beta<1$,  $2\leq l+\beta\leq k+1+\alpha$ and $n-s-\lambda<\delta<s-\lambda$, the map
\begin{equation}\label{eq.FH2}
P_{\lambda}: \Lambda_{\delta}^{l,\beta}(X) \longrightarrow  \Lambda_{\delta}^{l-2,\beta}(X)
\end{equation}
 is Fredholm with kernel consisting of $\rho^{-\lambda}f$ where $f$ is in the $L^2$-kernel of $\Delta_++s(n-s)$. 
 If further $\mathrm{Spec}(-\Delta_+)>s(n-s)$, then (\ref{eq.FH2}) is invertible. 
\end{lemma}
\begin{remark}\label{rmk.4}
If $\{g_+^{(i)}\}$ is a set of $C^{k,\alpha}$ conformally compact Poincar\'{e}-Einstein metric such that $\{\bar{g}^{(i)}=\rho^2g_+^{(i)}\}$ is compact in the $C^{k,\alpha}(\overline{X})$  topology and $\mathcal{Y}(M,[\hat{g}^{(i)}])\geq 0$ for all $i$, 
then for $0\leq \beta<1$, $2< l+\beta\leq k+\alpha$ and $n-s-\lambda<\delta<s-\lambda$, the map
\begin{equation*}
P_{\lambda}^{(i)}: \Lambda_{\delta}^{l,\beta}(X) \longrightarrow \Lambda_{\delta}^{l-2,\beta}(X)
\end{equation*}
 is invertible  for each $i$ and there is a constant $C>0$ independent of $i$ such that
 $$
 \|u\|_{\Lambda_{\delta}^{l,\beta}(X)} \leq C\|P_{\lambda}^{(i)}u\|_{\Lambda_{\delta}^{l-2,\beta}(X)}.  
 $$
\end{remark}

\subsection{Geodesic Normal Defining Function} 
While working with the boundary asymptotics, we usually choose the geodesic normal defining function. However, in our setting, it is no longer smooth.  A general regularity version of \cite[Lemma 5.1]{Le1} gives the following.

\begin{lemma}[Lee]
\label{lem.GDef}
There exists a boundary defining function $x\in C^{k-1,\alpha}(\overline{X})$	 such that
\begin{itemize}
\item[(1)] $x=\rho+O(\rho^2)$;
\item[(2)] $\tilde{g}=x^2g_+$ has a $C^{k-1,\alpha}(\overline{X})$ extension to $\overline{M}$ with $\tilde{g}|_{TM}=\bar{g}|_{TM}=\hat{g}$;
\item[(3)]	$\|dx\|^2_{\tilde{g}}\equiv 1$ in a collar neighbourhood of $M$.
\end{itemize}
\end{lemma}
\begin{lemma}\label{lem.GDef2}
Let $x$ and $\tilde{g}$ be given in Lemma \ref{lem.GDef}. Then
$$
\tilde{g}=dx^2 +\hat{g}+x^2g_2+O(x^{2+\epsilon}),
$$	
where $g_2=-\hat{A}$ and $\hat{A}_{ij}=\frac{1}{n-2}\left(\hat{R}_{ij}-\hat{J}\hat{g}_{ij}\right)$ is the Schouten tensor of $\hat{g}$. 
\end{lemma}
\begin{proof}
By a similar proof as \cite[Lemma 5.2]{Le1}, we have
$$
\begin{aligned}
	&  x\tilde{R}_{ij}+(n-1)\tilde{x}_{ij}+(\tilde{\Delta}x)\tilde{g}_{ij}=0,
	\\
	&  \tilde{\Delta}x=-\frac{1}{2n}x\tilde{R}=-x\tilde{R}_{ij}\tilde{x}^i\tilde{x}^j, 
\end{aligned}
$$
where $\tilde{\Delta}$ is the Laplace operator of $\tilde{g}$, and $\tilde{x}^i, \tilde{x}_{ij}$ denote the covariant derivatives w.r.t. $\tilde{g}$, and $\tilde{R}, \tilde{R}_{ij}$ are the scalar and Ricci curvatures of $\tilde{g}$. 
It is easy to see that $\tilde{x}_{ij}|_{M}=0$, which implies that $(M,\hat{g})$ is totally geodesic in $(\overline{X}, \tilde{g})$. Moreover,
$$
\hat{R}=\frac{n-1}{n}\tilde{R}|_M \quad \Leftrightarrow \quad \hat{J}=\tilde{J}|_M. 
$$
Hence by an identification of a collar neighborhood of $M$ with $[0,c)\times M$, we can write 
$$\tilde{g}=dx^2 +\hat{g}+x^2g_2+O(x^{2+\epsilon}).$$
Since $g_+=x^{-2}\tilde{g}$ satisfies  $Ric_{g_+}=-ng_+$,  the same calculation as in \cite{CDLS} gives that $g_2=-\hat{A}. $
\end{proof}

\vspace{0.2in}
\section{The Adapted and Fefferman-Graham compactifications}\label{sec.4}
In this section, we study some basic properties for the adapted and Fefferman-Graham compactifications for a Poincar\'{e}-Einstein manifold $(X,g_+)$, which were introduced in \cite{CC1} and \cite{FG2} respectively.  
Suppose $(X^{n+1}, g_+)$ ($n\geq 3$) is a $C^{k,\alpha}$ conformally compact Poincar\'{e}-Einstein manifold with $k\geq 3, 0<\alpha<1$. Let $\rho$ be a fixed smooth boundary defining function. Denote 
$$
\bar{g}=\rho^2g_+
\quad\textrm{and}\quad
\bar{g}|_{TM}=\hat{g}. 
$$


\subsection{Definition of $\rho_s$ and $\bar{g}_s$}

\subsubsection{Case 1: $s>\frac{n}{2}, s\neq n$}
In this case, if $\mathrm{Spec}(-\Delta_+)>s(n-s)$, let $v_s$ satisfy the following equation:
\begin{equation}\label{eq.def1}
-\Delta_+ v_s- s(n-s)v_s=0, \quad 
v_s=\rho^{n-s}\left(1+O(\rho^{\epsilon})\right)\ \textrm{for some $\epsilon>0$}.
\end{equation}
Then the \textit{adapted compactification} of $g_+$ is defined by
$$
\bar{g}_s=\rho_s^2 g_+, 
\quad\mathrm{where}\quad 
\rho_s=v_s^{\frac{1}{n-s}}.
$$
Notice that the adapted defining function $\rho_s\sim \rho$, which  may only have finite regularity. 

\begin{lemma}\label{lem.def1}
For $s>\frac{n}{2}, s\neq n$, if $\mathrm{Spec}(-\Delta_+)>s(n-s)$, then $\rho_s$ is uniquely determined by $\hat{g}$. 
\end{lemma}
\begin{proof}
First, with fixed $\rho$, (\ref{eq.def1}) has a solution. Notice that
$$
-\Delta_+\rho^{n-s}-s(n-s)\rho^{n-s} =-\rho^{n-s+1} \left[ (n-s)\bar{\Delta}\rho+s(n-s)T \right]. 
$$	
where $T$ is given in (\ref{eq.T}). Moreover, $\bar{\Delta}\rho, T\in C^{k-1,\alpha}(\overline{X}). $ Hence
$$
-\rho^{n-s+1} \left[ (n-s)\bar{\Delta}\rho+s(n-s)T \right]\in \rho^{n-s+1}C^{k-1,\alpha}(\overline{X})\subset \Lambda^{k-1,\alpha}_{n-s+1}(X). 
$$
If $s>\frac{n+1}{2}$, then $n-s+1\in (n-s,s)$. By Lemma \ref{lem.FH1}, there exists a unique $v_s'\in \Lambda^{k+1,\alpha}_{n-s+1}(X)$ such that 
$$
(-\Delta_+-s(n-s))(\rho^{n-s} +v_s')=0. 
$$ 
If $\frac{n+1}{2}\geq s>\frac{n}{2}$, then $n-s+1\geq s$. Take $\delta\in (n-s,s)$. Then by Lemma \ref{lem.FH1} again, there exists a unique $v_s''\in \Lambda^{k+1,\alpha}_{\delta}(X)$ such that
$$
(-\Delta_+-s(n-s))(\rho^{n-s} +v_s'')=0. 
$$

Second, $v_s>0$ in $X$. It is obvious that $v_s$ is positive near $M$. If $v_s$ has  negative  part in the interior, then by denoting $\Omega=\{v_s  < 0\}$, we have $\Omega\Subset X$ and
$$
\begin{cases}
-\Delta_+ v_s-s(n-s)v_s=0,	& \textrm{in $\Omega$,}
\\
v_s=0, &  \textrm{on $\partial \Omega$.}
\end{cases}
$$
This contradicts to $\mathrm{Spec}(-\Delta_+|_{\Omega})\geq \mathrm{Spec}(-\Delta_+)>s(n-s)$. So $v_s\geq 0$ and by strong maximum principle, $v_s>0$ in $X$. 

Third, the positivity also implies the uniqueness with given $\rho$. If there is a second solution $\tilde{v}_s$ satisfies (\ref{eq.def1}), then $\tilde{v}_s>0$ in $X$ and  
$$
\Delta_+\left(\frac{v_s}{\tilde{v}_s}\right) 
+ 2\left\langle d \log\tilde{v_s}, d\left(\frac{v_s}{\tilde{v}_s}\right)\right\rangle_{g_+}=0, 
\quad \left.\frac{v_s}{\tilde{v}_s}\right|_{M}=1.
$$
By the strong maximum principle, this implies that $v_s\equiv \tilde{v}_s$ in $X$. 

Finally, if there is another smooth defining function $\tilde{\rho}$ such that $\tilde{\rho}^2g_+|_{TM}=\hat{g}$ then we have $\tilde{\rho}\sim \rho$, which does not change the definition of $v_s$. 	Hence by above uniqueness, $\rho_s$ is only determined by the boundary representative $\hat{g}$. 
\end{proof}

\subsubsection{Case 2: $s=n$}
In this case, the equation(\ref{eq.def1}) is replaced by
 \begin{equation}\label{eq.def2}
 	-\Delta_+ w=n, \quad w=\log\rho+O(\rho^{\epsilon})\ \textrm{for some $\epsilon>0$}.
 \end{equation}
 Then the \textit{Fefferman-Graham compactification} of $g_+$ is defined by
 $$
 \bar{g}_F=\rho_F^2g_+, 
 \quad\mathrm{where}\quad 
 \rho_F=e^{w}.
 $$
Similarly the Fefferman-Graham defining function $\rho_F\sim \rho$, which  may only have finite regularity. 

 \begin{lemma}\label{lem.def2}
 The Fefferman-Graham defining function $\rho_F$ is uniquely determined by $\hat{g}$. 
 \end{lemma}
\begin{proof}
First,  (\ref{eq.def2}) has a solution. Notice that
$$
-\Delta _+(\log \rho)=n-\rho[\bar{\Delta}\rho+nT], 
$$
where 
$$
\rho[\bar{\Delta}\rho+nT]\in \rho C^{k-1,\alpha}(\overline{X})\subset \Lambda_{1}^{k-1,\alpha}(X).
$$
By Lemma \ref{lem.FH1}, there exists a unique $w'\in \Lambda_{1}^{k+1,\alpha}(X)$, such that
$$
-\Delta _+(\log \rho+w')=n.  
$$
Since $\mathrm{Spec}(-\Delta_+)>0$, then $w>0$ and is uniquely determined by  $\hat{g}$ by a similar proof of Lemma \ref{lem.def1}. 
\end{proof}

\subsubsection{Continuous Family}
In lots of discussions, we can view the Fefferman-Graham compactification as a special case of the adapted compactification in the following sense.

\begin{lemma}
If the Yamabe constant $\mathcal{Y}(M,[\hat{g}])\geq 0$, then the adapted defining function $\rho_s$ is a continuous family of boundary defining functions for $s>\frac{n}{2}$ such that
$$
\lim_{s\rightarrow n} \rho_s =\rho_F. 
$$\end{lemma}
\begin{proof}
For $\mathcal{Y}(M,[\hat{g}])\geq 0$, $v_s$ is a continuous family of solutions to (\ref{eq.def1}). While $s=n$, $v_s=1$. Take a derivative in $s$ at $s=n$ and denote 
$$
\left.\frac{dv_s}{ds}\right|_{s=n}=-w.
$$
Then $w$ satisfies (\ref{eq.def2}). By L'Hospital rule,
$$
\lim_{s\rightarrow n} \rho_s =\lim_{s\rightarrow n} v_s^{\frac{1}{n-s}}
=\lim_{s\rightarrow n} e^{\frac{1}{n-s}\ln v_s} =e^w=\rho_F. 
$$
\end{proof}

In the following of the paper, we will always assume 
$$
\mathcal{Y}(M,[\hat{g}])\geq 0.
$$
So that for all $s>\frac{n}{2}$, $\rho_s$ and $\bar{g}_s$ are well defined with given boundary representative $\hat{g}$. We will denote $\rho_n=\rho_F$ for consistency.

\subsection{Interior Regularity and Curvature Vanishing}
Notice that $g_+=\rho^{-2}\bar{g}$ has $C^{k,\alpha}$ regularity in the interior $X$. Therefore by the interior Schauder estimates for equations (\ref{eq.def1}) and (\ref{eq.def2}), for all $s>\frac{n}{2}$
$$
\rho_s\in C^{k+1,\alpha}(X),\quad \bar{g}_s\in C^{k,\alpha}(X).  
$$ 
This also can be seen from the proof of Lemma \ref{lem.def1} and \ref{lem.def2}. 

The adapted and Fefferman-Graham compactifications are introduced so that the compactified metric $\bar{g}_s$ possesses some certain geometric features. Let $\bar{R}_s$ be the scalar curvature of $\bar{g}_s$,  with $\bar{J}_s=\frac{1}{2n}\bar{R}_s$,  and $\bar{Q}_{2N}^{\bar{g}_s}$ its $2N$-th order Q-curvautre. Let $\bar{\Delta}_s$ be the Laplace-Beltrami operator for $\bar{g}_s$.
Direct calculation shows that in $X$, 
\begin{equation}\label{eq.bdf}
\bar{\Delta}_s \rho_s=-s \bar{T}_s,
\end{equation}
where
\begin{equation}
\bar{T}_s=	 \rho_s^{-1}\left(1-|d\rho_s|^2_{\bar{g}_s}\right),
\end{equation}
and
\begin{equation}\label{eq.J}
\bar{J}_s=\frac{2s-n-1}{2} \rho_s^{-2}\left(1-|d \rho_s|^2_{\bar{g}_s}\right).
\end{equation}
Therefore $\bar{T}_s, \bar{J}_s\in C^{k,\alpha}(X)$ and $\bar{Q}_{2N}^{\bar{g}_s}\in C^{k-2N+2,\alpha}(X)$.

\begin{lemma}\label{lem.Q}
If $s=\frac{n}{2}+N-\frac{1}{2}$ for some integer $1\leq N \leq [\frac{k+1}{2}]$, then $(\overline{X}, \bar{g}_s)$ has vanishing $\bar{Q}^{\bar{g}_s}_{2N}$ curvature in $X$. 
\end{lemma}
\begin{proof}
If 	$s=\frac{n}{2}+N-\frac{1}{2} \neq n$, then (\ref{eq.def1}) implies that
$$
\begin{aligned}
& P_{2N}^+(v_s)=P_{2N}^+\left(\rho^{\frac{n+1}{2}-N}_s\right)=0, 
\\ \Longrightarrow\quad
& \bar{P}^{\bar{g}_s}_{2N}(1) = \rho^{-\frac{n+1}{2}-N}P_{2N}^+\left(\rho^{\frac{n+1}{2}-N}_s\right)=0,
\\ \Longrightarrow\quad
& \bar{Q}^{\bar{g}_s}_{2N}=0, \quad \textrm{in $X$}. 
\end{aligned}
$$
If $s=\frac{n}{2}+N-\frac{1}{2} = n$, i.e. $2N=n+1$, by \cite[Theorem 1.1-1,2]{Go1}, the $(n+1)$-order GJMS operator and Q-curvature of $(X, g_+)$ is 
$$
\begin{gathered}
P^+_{n+1}= \left(-\Delta_+-(n-1)\right)\left(-\Delta_+-3(n-3)\right)\cdots 
\left(-\Delta_+-(n-2)2\right)\left(-\Delta_+\right), 
\\
Q^+_{n+1}=(-1)^\frac{n+1}{2}n!. 
\end{gathered}
$$
Then (\ref{eq.def2}) implies that
$$
0=P^+_{n+1}w+Q^+_{n+1}=\bar{Q}^{\bar{g}_s}_{n+1} e^{(n+1)w}
\quad\Longrightarrow \quad
\bar{Q}^{\bar{g}_s}_{n+1}=0, \ \textrm{in $X$}. 
$$
\end{proof}

We want to point out that in \cite{CC1} the authors  considered the metric measure space and defined the more general $\bar{Q}_{2N,\phi}^m$ curvatures, and proved that if $s=\frac{n+2N-1-m}{2}$, then $\bar{Q}_{2N,\phi}^m=0$ for $\bar{g}_s$. In later applications, the equation $\bar{Q}^{\bar{g}_s}_{2N}=0$ will help us to lift the regularity of $\bar{J}_s$. However, $\bar{Q}_{2N,\phi}^m=0$ can not. So we avoid going to the details of $\bar{Q}_{2N,\phi}^m$ here and ask the readers to refer to \cite{CC1} if interested.

\subsection{Global Regularity of $\rho_s$ and $\bar{g}_s$}
It is obvious that the regularity of $\bar{g}_s$ depends on both the initial regularity of $\bar{g}$ and the parameter $s$.  Even if $g_+$ is $C^{\infty}$ conformally compact and the geodesic normal defining function $x$ is smooth in Lemma \ref{lem.GDef}, we know that $v_s$ takes the form $v_s=x^s F+x^{n-s}G$ for some $F,G\in C^{\infty}(\overline{X})$ if  $2s-n$ is fractional. And the regularity of $\rho_s$ is stoped by this fractional order. Therefore it is nature to conjecture that the regularity of $\rho_s$ and $\bar{g}_s$ is given by the minimum of the background regularity of $\bar{g}$ and the fractional order $2s-n$. We give a strict proof here.

\subsubsection{Case 1:  $s>\frac{n}{2}, s\neq n$}
In this case, denote
$$
\bar{g}_s=\rho_s^2g_+=\phi_s^{\frac{2}{n-s}} \bar{g}, \quad \phi_s=\left(\frac{\rho_s}{\rho}\right)^{n-s}. 
$$
Then $v_s=\phi_s\rho^{n-s}$ and (\ref{eq.def1}) implies that
$$
\Delta_+v_s+s(n-s)v_s=\rho^{n-s}P(\phi_s),
$$ 
where $P=\rho^{s-n}(\Delta_++s(n-s))\rho^{n-s}$. Let us write $P$ in background metric $\bar{g}$: 
$$
P\phi_s= \rho^2\bar{\Delta}\phi_s -(2s-n-1)\rho\langle d\rho, d\phi_s\rangle_{\bar{g}}
+(n-s)\rho\left(\bar{\Delta}\rho +sT\right) \phi_s, 
$$
where
$$
T=\frac{1-|d\rho|^2_{\bar{g}}}{\rho}
\in C^{k-1,\alpha}(\overline{X}).
$$
Therefore, $\phi_s$ satisfies
\begin{equation}\label{eq.phi}
\begin{cases}
P \phi_s=0, & \textrm{in $X$}, 
\\
\phi_s=1,  & \textrm{on $M$}. 
\end{cases}
\end{equation}

We begin by constructing an asymptotical solution for (\ref{eq.phi}). 

\begin{lemma}\label{lem.reg0}
	For $s>\frac{n}{2}, s\neq n$, there exists an approximate solution $\tilde{\phi}_s   \in  C^{k,\alpha}(\overline{X})$ such that
	$$
	P(\tilde{\phi}_s) = \tilde{F}, \quad \tilde{F}\in \Lambda^{k-1,\alpha}_{\delta}(X). 
	$$
	where $\delta=k+\alpha$ if $k+\alpha\leq 2s-n$ or $2s-n\notin\mathbb{N}$; 
	$\delta=2s-n$ if $k+\alpha>2s-n$ and $2s-n\in \mathbb{N}$. 
	
	Moreover, there exists a constant $C$ only depending on $k,\alpha$ and the $C^{k,\alpha}(\overline{X})$ norm of $\bar{g}$ such that
	$$
	\|\tilde{\phi}_s\|_{  C^{k,\alpha}(\overline{X})} \leq C, 
	\quad
	\| \tilde{F}\|_{\Lambda^{k-1,\alpha}_{\delta}(X)}\leq C;
	$$
	and as $\rho\rightarrow 0$, 	$\tilde{\phi}_s$ has asymptotic expansion
	$$
	\tilde{\phi}_s=1+\rho \phi_{s,1}+\cdots +\rho^l  \phi_{s,l},
	$$
where  $l=k$ if $k+\alpha\leq 2s-n$ or $2s-n\notin\mathbb{N}$; $l=2s-n-1$ if $k+\alpha>2s-n$ and $2s-n\in \mathbb{N}$. 
Here $\phi_{s,j}\in C^{k-j,\alpha}(M)$ is determined by the derivatives of $\bar{g}$ up to  $j$-th order and
	$$
	\|\phi_{s,j}\|_{ C^{k-j,\alpha}(M)}\leq C. 
	$$
\end{lemma}
\begin{proof}
Let $s=\frac{n}{2}+\gamma$. 
First $\bar{g}$ is $C^{k,\alpha}(\overline{X})$ and $\rho\in C^{\infty}(\overline{X})$, which together implies that the coefficients in $P$ are in $C^{k-1,\alpha}(\overline{X})$. 
By Theorem \ref{thm.ext}, for any $\phi_{s,j}\in C^{k-j,\alpha}(M)$, we can take $\tilde{\phi}_{s,j}\in C^{k,\alpha}(\overline{X})$   as  an extension of  $\rho^j\phi_{s,j}$ such that $\tilde{\phi}_{s,j}\sim \rho^j\phi_{s,j}$ as $\rho \rightarrow 0$. Moreover, there exists some constant $C'$ only depending on $k,j,\alpha$ such that
$$
\|\tilde{\phi}_{s,j}\|_{C^{k,\alpha}(\overline{X})} \leq C'\|\phi_{s,j}\|_{C^{k-j,\alpha}(M)}. 
$$
Notice that under the background coordinates $(\rho,\theta)$, $P$ is in the form
$$
P=\rho\left(a^{\mu\nu}\rho\partial_{\mu}\partial_{\nu}+b^{\mu}\rho \partial_{\mu}+b'^{\mu} \partial_{\mu}+c\right),
$$
where $a^{\mu\nu}, b'^{\mu}\in C^{k,\alpha}(\overline{X})$ and $b^{\mu},c \in C^{k-1,\alpha}(\overline{X})$. 
By (3) in Theorem \ref{thm.ext}, $ \rho^{-1}P \tilde{\phi}_{s,j}\in  C^{k-1,\alpha}(\overline{X}) $, and there exists some constant $C''$ only depending on the $C^{k,\alpha}(\overline{X})$ norm of  $\bar{g}$,   such that
$$
\| \rho^{-1}P \tilde{\phi}_{s,j}\|_{C^{k-1,\alpha}(\overline{X})} \leq C''\|\phi_{s,j}\|_{C^{k-j,\alpha}(M)}. 
$$
As $\rho\rightarrow 0$, 
$$
P(\tilde{\phi}_{s,j})= -j(2\gamma-j)\rho^j\phi_{s,j}(\theta)+O(\rho^{j+\epsilon}),
$$
where $\epsilon=1$ if $j<k$ and $\epsilon=\alpha$ if $j=k$. 

Now we are ready to construct the asymptotical solution $\tilde{\phi}_s$.
Take $\tilde{\phi}_{s,0}=\phi_{s,0}=1$. Then
$$
P(1)=(n-s)\rho\left(\bar{\Delta}\rho +sT\right)=\tilde{F_0}\sim\rho F_0, 
$$
where $\tilde{F}_0\in \rho C^{k-1,\alpha}(\overline{X}) $ and $F_0=(n-s)\left(\bar{\Delta}\rho +sT\right)|_{M}\in C^{k-1,\alpha}(M)$. Moreover, for some constant $C_1$ only depending on the $C^{k,\alpha}(\overline{X})$ norm of  $\bar{g}$, 
$$
\|\rho^{-1}\tilde{F}_0\| _{C^{k-1,\alpha}(\overline{X})} \leq C_1, 
\quad
\|F_0\|_{C^{k-1,\alpha}(M)} \leq C_1. 
$$
Choose $\phi_{s,1}=(2\gamma-1)^{-1}F_0$.  Let $\tilde{\phi}_{s,1}\in C^{k,\alpha}(\overline{X})$ be an extension of $\rho\phi_{s,1}$ by Theorem \ref{thm.ext}. Then
$$
P(1+\tilde{\phi}_{s,1})=\tilde{F_1}\sim \rho^2 F_1, 
$$
where $\tilde{F}_1\in \rho C^{k-1,\alpha}(\overline{X})$ and $F_1\in C^{k-2,\alpha}(M)$. And similar as above for some constant $C_2$ only depending on the $C^{k,\alpha}(\overline{X})$ norm of  $\bar{g}$, 
$$
\| \rho^{-1}\tilde{F}_1\| _{C^{k-1,\alpha}(\overline{X})} \leq C'_2, 
\quad
\|F_1\|_{C^{k-2,\alpha}(M)} \leq C_2.
$$
Choose $\phi_{s,2}=[2(2\gamma-2)]^{-1}F_1$.

If $k+\alpha\leq 2\gamma$ or $2\gamma\notin\mathbb{N}$, we can repeat this progress until order $k$ such that  $\phi_{s,k}\in C^{0,\alpha}(M)$. Let $\tilde{\phi}_{s,k}\in C^{k,\alpha}(\overline{X})$ be an extension of $\rho^k\phi_{s,k}$ by Theorem \ref{thm.ext}. 
$$
P(1+\tilde{\phi}_{s,1}+\cdots+\tilde{\phi}_{s,k})=\tilde{F_k}=O(\rho^{k+\alpha} ), 
$$
where $\tilde{F}_k\in \rho C^{k-1,\alpha}(\overline{X})$. And similar as above for some constant $C_k$ only depending on the $C^{k,\alpha}(\overline{X})$ norm of  $\bar{g}$, 
$$
\|\rho^{-1}\tilde{F}_k\| _{C^{k-1,\alpha}(\overline{X})} \leq C_k.
$$
Finally, we take $\tilde{\phi}_s=1+\tilde{\phi}_{s,1}+\cdots +\tilde{\phi}_{s,k}$ and $\tilde{F}=\tilde{F}_k$. 
By Lemma \ref{lem.FS2}, $\rho^{-1}\tilde{F}_k\in  C^{k-1,\alpha}(\overline{X})$ and $\rho^{-1}\tilde{F}_k=O(\rho^{k-1+\alpha})$ imply that
$$
\rho^{-1}\tilde{F}\in C^{k-1,\alpha}_{(k-1+\alpha)}(\overline{X})\subset  \Lambda^{k-1,\alpha}_{k-1+\alpha}(X)
\quad\Longrightarrow \quad \tilde{F}\in  \Lambda^{k-1,\alpha}_{k+\alpha}(X).
$$
By Lemma \ref{lem.FS3}, \ref{lem.FS2},  the inclusion  $C^{k-1,\alpha}_{(k-1+\alpha)}(\overline{X})\hookrightarrow \Lambda^{k-1,\alpha}(\overline{X})$ is continuous and $C^{k-1,\alpha}_{(k-1+\alpha)}$ is a closed subspace of  $C^{k-1,\alpha}(\overline{X})$. Hence there exists a constant $\tilde{C}$ only depending on $X$ such that
$$
\|\tilde{F}\|_{\Lambda^{k-1,\alpha}_{k+\alpha}(X)}=\|\rho^{-1}\tilde{F} \|_{\Lambda^{k-1,\alpha}_{k-1+\alpha}(X)} \leq 
\tilde{C} \|\rho^{-1}\tilde{F} \|_{C^{k-1,\alpha}(\overline{X})}\leq \tilde{C}C_k. 
$$

If $k+\alpha>2\gamma$ and $2\gamma\in \mathbb{N}$, then we stop at the order $2\gamma-1$. And
$$
P(1+\tilde{\phi}_{s,1}+\cdots+\tilde{\phi}_{s,2\gamma-1})=\tilde{F}_{2\gamma-1}=O(\rho^{2\gamma} ), 
$$
where $\tilde{F}_{2\gamma-1}\in \rho C^{k-1,\alpha}(\overline{X})$ and
$$
\|\rho^{-1}\tilde{F}_{2\gamma-1}\| _{C^{k-1,\alpha}(\overline{X})} \leq C_{2\gamma-1}.
$$
Finally, we take $\tilde{\phi}_s=1+\tilde{\phi}_{s,1}+\cdots +\tilde{\phi}_{s,2\gamma-1}$ and $\tilde{F}=\tilde{F}_{2\gamma-1}$. Then
$$
\rho^{-1}\tilde{F}\in C^{k-1,\alpha}_{(2\gamma-1)}(\overline{X})\subset  \Lambda^{k-1,\alpha}_{2\gamma-1}(X)
\quad\Longrightarrow \quad \tilde{F}\in  \Lambda^{k-1,\alpha}_{2\gamma}(X).
$$
Moreover, 
$$
\|\tilde{F}\|_{\Lambda^{k-1,\alpha}_{2\gamma}(X)}=\|\rho^{-1}\tilde{F} \|_{\Lambda^{k-1,\alpha}_{2\gamma-1}(X)} \leq 
\tilde{C} \|\rho^{-1}\tilde{F} \|_{C^{k-1,\alpha}(\overline{X})}\leq \tilde{C}C_{2\gamma-1}. 
$$
We finish the proof. 
\end{proof}

\begin{remark}\label{rmk.4}
While $k+\alpha>2s-n$ and $2s-n=2\gamma\in\mathbb{N}$, actually we can construct a full asymptotical solution such that $P(\tilde{\phi})=O(\rho^{k+\alpha-\epsilon})$ if we allow logrithmic terms in. However, here it does not help to discuss the global H\"{o}lder regularity, so we omit it. 
\end{remark}

\begin{lemma}\label{lem.reg1}
For $s>\frac{n}{2}, s\neq n$, the equation (\ref{eq.phi}) has a unique solution  $\phi_s\in C^{l,\beta}(\overline{X})$,   where $l+\beta=k+\alpha$ if $k+\alpha<2s-n$; $l+\beta<2s-n$ if $k+\alpha\geq 2s-n$.  
Moreover, there exists some constant $C>0$ only depending on the $C^{k,\alpha}(\overline{X})$ norm of $\bar{g}$ such that 
$$
\|\phi_s\|_{C^{l,\beta}(\overline{X})}\leq C. 
$$
\end{lemma}
\begin{proof}
Denote $s=\frac{n}{2}+\gamma$.  
Take  $\tilde{\phi}_s$  the approximate solution in Lemma \ref{lem.reg0}. Now we modify it to be a true solution. 
First, consider $k+\alpha<2\gamma$. In this case, 
$$
P\tilde{\phi}_s=\tilde{F}\in \Lambda_{k+\alpha}^{k-1,\alpha}(X). 
$$
By Lemma \ref{lem.FH2},
$$
P: \Lambda_{k+\alpha}^{k+1,\alpha}(X) \longrightarrow \Lambda_{k+\alpha}^{k-1,\alpha}(X), 
$$
is invertible. Hence we can solve a unique 
$\tilde{\phi}'_s \in \Lambda_{k+\alpha}^{k+1,\alpha}(X)$ such that
$$
P(\tilde{\phi}_s+\tilde{\phi}_s')=0. 
$$
Moreover, there exists a constant $C$ only depending on the $C^{k,\alpha}(\overline{X})$ norm of $\bar{g}$ such that
$$
\|\tilde{\phi}_s'\|_{\Lambda_{k+\alpha}^{k+1,\alpha}(X)} \leq C \|\tilde{F}\|_{\Lambda_{k+\alpha}^{k-1,\alpha}(X)}. 
$$
By Lemma \ref{lem.FS2}, $\Lambda_{k+\alpha}^{k+1,\alpha}(X)\hookrightarrow C^{k,\alpha}(\overline{X})$ is continuous. Hence for some constant $\tilde{C}$ only depending on $X$, we have
$$
\|\tilde{\phi}_s'\|_{C^{k,\alpha}(\overline{X})} \leq \tilde{C} C \|\tilde{F}\|_{\Lambda_{k+\alpha}^{k-1,\alpha}(X)}. 
$$

Second, if $k+2\alpha\geq 2\gamma$, then
$$
P\tilde{\phi}_s=\tilde{F}\in \Lambda_{2\gamma}^{k-1,\alpha}(X)\subset \Lambda_{l+\beta}^{k-1,\alpha}(X), 
$$
for $l\geq 0$, $0<\beta<1$ such that $l+\beta<2\gamma$. 
By Lemma \ref{lem.FH2} again
$$
P: \Lambda_{l+\beta}^{k+1,\alpha}(X) \longrightarrow \Lambda_{l+\beta}^{k-1,\alpha}(X), 
$$
is invertible.  Hence we can solve a unique 
$\tilde{\phi}'_s \in \Lambda_{l+\beta}^{k+1,\alpha}(X)$ such that
$$
P(\tilde{\phi}_s+\tilde{\phi}_s')=0. 
$$
Moreover, there exists a constant $C$ only depending on the $C^{k,\alpha}(\overline{X})$ norm of $\bar{g}$ and $C'$ only depending on the maximum value of $\rho$ such that
$$
\|\tilde{\phi}_s'\|_{\Lambda_{l+\beta}^{k+1,\alpha}(X)} 
\leq C \|\tilde{F}\|_{\Lambda_{l+\beta}^{k-1,\alpha}(X)}
\leq CC' \|\tilde{F}\|_{\Lambda_{2\gamma}^{k-1,\alpha}(X)}.
$$
By Lemma \ref{lem.FS2}, $\Lambda_{l+\beta}^{k+1,\alpha}(X)\hookrightarrow C^{l,\beta}(\overline{X})$ is continuous. Hence for some constant $\tilde{C}$ only depending on $X$, we have
$$
\|\tilde{\phi}_s'\|_{C^{l,\beta}(\overline{X})} \leq \tilde{C} CC' \|\tilde{F}\|_{\Lambda_{2\gamma}^{k-1,\alpha}(X)}. 
$$

Combing above discussion with Lemma \ref{lem.reg0}, we finish the proof. 
\end{proof}

\begin{lemma}\label{lem.reg2}
If $s=\frac{n}{2}+N-\frac{1}{2}$ for some integer $N\geq 1$, then $\phi_s\in C^{k,\alpha}(\overline{X})$ 
and there exists some constant $C>0$ only depending on the $C^{k,\alpha}(\overline{X})$ norm of $\bar{g}$ such that 
$$
\|\phi_s\|_{C^{k,\alpha}(\overline{X})}\leq C. 
$$
\end{lemma}
\begin{proof}
First, if $N=1$, then $s=\frac{n+1}{2}$ and (\ref{eq.R}) implies that
$$
\rho\bar{\Delta}\rho+s(1-|d\rho|^2_{\bar{g}})=-\bar{J} \rho^2, 
$$
where $\bar{J}=\frac{1}{2n}\bar{R}\in C^{k-2,\alpha}(\overline{X})$. 
Hence equation (\ref{eq.phi}) becomes 
$$
\begin{cases}
	\bar{\Delta}\phi_s -\frac{n-1}{2} \bar{J}\phi_s=0, & \textrm{in $X$}
	\\
	\phi_s=1, & \textrm{on $M$}.
\end{cases}
$$
By classical Schauder estimates, we have $\phi_s\in C^{k,\alpha}(\overline{X})$.

Next, consider $N\geq 2$. 
Notice that if $k+\alpha<2N-1$, it is already proved in Lemma \ref{lem.reg1}. So we only need to consider the case $k+\alpha \geq 2N-1$. First choose $l=2N-2$. Then $l+\alpha<2N-1$ and by Lemma \ref{lem.reg1}, we have
$$
\phi_s\in C^{k,\alpha}(X)\cap C^{l,\alpha}(\overline{X}). 
$$
Recall that by Lemma \ref{lem.Q}, in $X$,
\begin{equation}\label{eq.1}
Q^{\bar{g}_s}_{2N}=\Delta_{\bar{g}_s}^{N-1} \bar{J}_s+L.O.T. =0, 
\end{equation}
where $L.O.T$ consists of derivatives of $\bar{g}_s$ of order $\leq 2N-2$. Hence
$$
L.O.T. \in C^{k-2N+2,\alpha}(X)\cap C^{0,\alpha}(\overline{X}). 
$$
By Lemma \ref{lem.asym1}, on the boundary $\bar{J}_s|M\in C^{k-2,\alpha}(M)$. Hence by global Schauder estimate,
$$
\bar{J}_s\in C^{l,\alpha}(\overline{X}). 
$$
Consider the conformal transformation from $\bar{g}$ to $\bar{g}_s$, where $\bar{g}_s=\phi_s^{\frac{2}{n-s}}\bar{g}$. Then
\begin{equation}\label{eq.2}
-\frac{2n}{n-s}\frac{\bar{\Delta}\phi_s}{\phi_s}-\frac{n(n-3)}{n-s}\frac{|d\phi_s|^2_{\bar{g}}}{\phi_s^2}+\bar{R} =\phi_s^{\frac{2}{n-s}}\bar{R}_s.
\end{equation}
Notice that by the definition of $\phi_s$ and Lemma \ref{lem.reg1}, $c<\phi_s<C$ for some $0<c<C$ only depending on the $C^{k,\alpha}(\overline{X})$ norm of $\bar{g}$.  
Now $k-2\geq l-1$, $0<c<\phi_s<C$, $\phi_s, \bar{R}_s \in  C^{l,\alpha}(\overline{X})$, $|d\phi_s|^2_{\bar{g}}\in C^{l-1,\alpha}(\overline{X})$ and $\bar{R}\in  C^{k-2,\alpha}(\overline{X})$. These  imply that 
$\phi_s\in C^{l+1,\alpha}(\overline{X})$ by the global Schauder estimates. If $k-2\geq l$, using $|d\phi_s|^2_{\bar{g}}\in C^{l,\alpha}(\overline{X})$ again in equation (\ref{eq.2}), we have $\phi_s\in C^{l+2,\alpha}(\overline{X})$. 
Repeat this  lifting  progress and apply the global Schauder estimates  to (\ref{eq.1}) and (\ref{eq.2}), until the regularity lifting is stoped by $C^{k,\alpha}(\overline{X})$ assumption on $\bar{g}$. Finally we get  $\phi_s\in C^{k,\alpha}(\overline{X})$. And the estimates of $\|\phi_s\|_{C^{k,\alpha}(\overline{X})}$ is directly from Lemma \ref{lem.reg1} and the global Schauder estimates for (\ref{eq.1}) and (\ref{eq.2}).
\end{proof}

\begin{remark}\label{rmk.6}
	If $\{g_+^{(i)}\}$ is a set of $C^{k,\alpha}$ conformally compact Poincar\'{e}-Einstein metric such that $\{\bar{g}^{(i)}=\rho^2g_+^{(i)}\}$ is compact in the $C^{k,\alpha}(\overline{X})$  topology and $\mathcal{Y}(M,[\hat{g}^{(i)}])\geq 0$ for all $i$, then Lemma \ref{lem.reg0}, \ref{lem.reg1}, \ref{lem.reg2} hold with uniform estimates, i.e.
	$$
	\|\phi_s^{(i)}\|_{C^{l,\beta}(\overline{X})} \leq C, 
	$$
	where $C>0$  is independent of $i$.
	Here $l+\beta=k+\alpha$ if $k+\alpha<2s-n$ or $2s-n=2N-1$ for some positive integer $N$; $l+\beta<2s-n$ if $k+\alpha\geq 2s-n$ and $2s-n$ is not an odd positive integer. 
\end{remark}

\subsubsection{Case 2: $s=n$}
For $s=n$, denote
$$
\bar{g}_n=\rho_n^2g_+=e^{2\varphi}\bar{g}, \quad \varphi=w-\log\rho.
$$
Then $\varphi$ satisfies
\begin{equation}\label{eq.vphi}
\begin{cases}
\Delta_+\varphi
+\rho\left(\bar{\Delta}\rho +nT\right) =0, & \textrm{in $X$}, 
\\
\varphi=0,  & \textrm{on $M$},
\end{cases}
\end{equation}
where 
$$
\Delta_+\varphi=\rho^2\bar{\Delta}\varphi -(n-1)\rho\langle d\rho, d\varphi\rangle_{\bar{g}}.
$$ 

A similar proof of Lemma \ref{lem.reg0}, \ref{lem.reg1}, \ref{lem.reg2} gives the following.

\begin{lemma}\label{lem.reg5}
	There exists an approximate solution $\tilde{\varphi}\in  C^{k,\alpha}(\overline{X})$ such that
	$$
	\Delta_+\tilde{\varphi}= \tilde{F}, \quad \tilde{F}\in \Lambda^{k-1,\alpha}_{\delta}(X). 
	$$
	where $\delta=k+\alpha$ if $k+\alpha\leq n$; $\delta=n$ if $k+\alpha> n$. 
	Moreover, there exists a constant $C$ only depending on $k,\alpha$ and the $C^{k,\alpha}(\overline{X})$ norm of $\bar{g}$ such that
	$$
	\|\tilde{\varphi}\|_{  C^{k,\alpha}(\overline{X})} \leq C, 
	\quad
	\| \tilde{F}\|_{\Lambda^{k-1,\alpha}_{\delta}(X)}\leq C;
	$$
	and as $\rho\rightarrow 0$, 
	$$
	\tilde{\varphi}=1+\rho \varphi_{1}+\cdots +\rho^l  \varphi_{l},
	$$
	where $l=k$  if $k+\alpha\leq n$; $l=n-1$ if $k+\alpha>n$. 
	Here $\varphi_{j}\in C^{k-j,\alpha}(M)$ is determined by the derivatives of $\bar{g}$ up to  $j$-th order and
	$$
	\|\varphi_{j}\|_{ C^{k-j,\alpha}(M)}\leq C. 
	$$
\end{lemma}

\begin{lemma}\label{lem.reg3}
Let $l+\beta=k+\alpha$ if $k+\alpha<n$ and $l+\beta<n$ if $k+\alpha\geq  n$. 
Then $\varphi\in C^{l,\beta}(\overline{X})$,  and there exists some constant $C>0$ only depending on the $C^{k,\alpha}(\overline{X})$ norm of $\bar{g}$ such that 
$$
\|\varphi\|_{C^{l,\beta}(\overline{X})}\leq C. 
$$
\end{lemma}
\begin{lemma}\label{lem.reg4}
If $n$ is odd, then $\varphi \in C^{k,\alpha}(\overline{X})$ and there exists some constant $C>0$ only depending on the $C^{k,\alpha}(\overline{X})$ norm of $\bar{g}$ such that 
$$
\|\varphi\|_{C^{k,\alpha}(\overline{X})}\leq C. 
$$
\end{lemma}

\begin{remark}\label{rmk.7}
	If $\{g_+^{(i)}\}$ is a set of $C^{k,\alpha}$ conformally compact Poincar\'{e}-Einstein metric such that $\{\bar{g}^{(i)}=\rho^2g_+^{(i)}\}$ is compact in the $C^{k,\alpha}(\overline{X})$  topology and $\mathcal{Y}(M,[\hat{g}^{(i)}])\geq 0$ for all $i$, then Lemma \ref{lem.reg5}, \ref{lem.reg3}, \ref{lem.reg4} hold with uniform estimates, i.e.
	$$
	\|\varphi^{(i)}\|_{C^{l,\beta}(\overline{X})} \leq C,
	$$
	where $C>0$  is independent of $i$. 
	Here $l+\beta=k+\alpha$ if $k+\alpha<n$ or $n$ is odd;  $l+\beta<n$ if $k+\alpha\geq n$ and $n$ is even. 
\end{remark}



\vspace{0.2in}
\section{Geometric Positivity Results for $\bar{g}_s$.}\label{sec.5}

In this section we investigate some more geometric features for the adapted and Fefferman-Graham compactifications. Suppose $(X^{n+1}, g_+)$ ($n\geq 3$) is a $C^{k,\alpha}$ conformally compact Poincar\'{e}-Einstein manifold with $k\geq 3$, $0<\alpha<1$. For $s>\frac{n}{2}$, $\bar{g}_s$ is the adapted (including Fefferman-Graham) compactification defined in last section. Denote 
$$
\bar{g}_s=\rho_s^2g_+
\quad\textrm{and}\quad
\bar{g}_s|_{TM}=\hat{g}. 
$$
Let $H_s$ be the boundary mean curvature of $\bar{g}_s$ if it is well defined. 
Let $\hat{R}$ be the scalar curvature of $\hat{g}$ and  $\hat{J}=\frac{1}{2n}\hat{R}$.

\subsection{Boundary Asymptotics}
First recall the geodesic normal defining function $x$  from Lemma \ref{lem.GDef}. Notice that $x\in C^{k-1,\alpha}(\overline{X})$ is determined by a given boundary representative $\hat{g}$. Then
$$
\tilde{g}=x^2g_+
\quad\textrm{and}\quad 
\tilde{g}|_{TM}=\bar{g}|_{TM}=\bar{g}_s|_{TM}=\hat{g}.
$$ 
Recall that
$$
\bar{J}_s =\frac{2s-n-1}{2} \left(\frac{1-|d\rho_s|^2_{\bar{g}_s}}{\rho_s^2}\right),
\quad
\bar{T}_s =\frac{1-|d\rho_s|^2_{\bar{g}_s}}{\rho_s}.
$$

\begin{lemma}\label{lem.asym1}
For $s>\frac{n}{2}+1$, we have
$$
\rho_s=x\left( 1-\frac{1}{2(2s-n-2)}\hat{J} x^2 + O(x^{2+\epsilon})\right), 
$$
and
$$
\bar{J}_s|_{M}=\frac{2s-n-1}{2s-n-2}\hat{J}.
$$

\end{lemma}
\begin{proof}
For $s>\frac{n}{2}+1,s\neq n$, we work on the equation (\ref{eq.def1}). By Lemma \ref{lem.GDef2} and the calculation in its proof, we have
$$
-\Delta_+ x^{n-s} -s(n-s)x^{n-s}=-x^{n-s+1}[(n-s)\tilde{\Delta}x+s(n-s)\tilde{T}], 
$$
where $$\tilde{T}=\frac{1-|dx|^2_{\tilde{g}}}{x}.$$
By the definition of $x$, $\tilde{T}\equiv 0$ in a neighbourhood of $M$ and hence for some $c>0$ small enough,
$$
-x^{n-s+1}[(n-s)\tilde{\Delta}x+s(n-s)\tilde{T}]=x^{n-s+2}\left(\frac{n-s}{2n}\tilde{R}\right), \quad \textrm{in $X_c$}. 
$$
Here $\tilde{R}\in C^{k-3,\alpha}(\overline{X})$. 
Notice that
$$
\begin{aligned}
&\ \left[-\Delta_+ -s(n-s)\right](x^{r}f) 
\\
=&\  x^r\left[(r-s)(r+s-n)f -r(r-n)(1-|dx|^2_{\tilde{g}})f +rx\tilde{\Delta}x f+(2r-n+1)x\langle dx,df\rangle_{\tilde{g}} +x^2\tilde{\Delta}f 
\right].
\end{aligned}
$$
Take $r=n-s+2$ and $f\in C^{k-2,\alpha}(\overline{X})$ be an extension of 
$$
f_0=-\frac{n-s}{2(2s-n-2)} \hat{J}\in C^{k-2,\alpha}(M)
$$
according to Theorem \ref{thm.ext}. 
Then
$$
\left[-\Delta_+ -s(n-s)\right](x^{n-s}+x^{n-s+2}f)=x^{n-s+3} F, 
\quad F\in  C^{k-3,\alpha}(\overline{X}). 
$$
By Lemma \ref{lem.FH1}, there exists $v'\in \Lambda_{0}^{k-1,\alpha}(X)$ such that 
$$
\left[-\Delta_+ -s(n-s)\right](x^{n-s}+x^{n-s+2}f+x^{n-s+2+\epsilon}v')=0,
$$
So $v_s$ has asymptotics 
$$
v_s=x^{n-s}\left(1-\frac{n-s}{2(2s-n-2)}\hat{J} x^{2} + O(x^{2+\epsilon})\right)
$$	
which implies the asymptotics of $\rho_s$. 


If $s=n$, we work on the equation (\ref{eq.def2}). Then similar proof as above gives the asymptotical expansion of $\rho_n$, which also can be read from the continuity of $\rho_s$ in parameter $s$.
\end{proof}


\begin{lemma}\label{lem.asym2}
For $s=\frac{n+1}{2}$,  we have
$$
\rho_s=x\left(1-\frac{1}{n}H_sx+\frac{1}{2}\hat{J}x^2+O(x^{2+\epsilon})\right), 
$$	
and
$$
\bar{T}_s|_{M}=\frac{2}{n}H_s. 
$$
\end{lemma}
\begin{proof}
By Lemma \ref{lem.reg2}, we know  for $s=\frac{n+1}{2}$, $\rho_s\in C^{k,\alpha}(\overline{X})$ and $\bar{g}_s\in C^{k,\alpha}(\overline{X})$. Assume 
$$
\rho_s=x\psi, \quad \psi=1+x \rho_{s,1}+x^2 \rho_{s,2}+O(x^{2+\epsilon}). 
$$
By Lemma \ref{lem.GDef2}, $(M,\hat{g})$  is totally geodesic in $(\overline{X}, \tilde{g})$ and
$$
\bar{g}_s=\psi^2 \tilde{g}=\left(1+x \rho_{s,1}+x^2 \rho_{s,2}+O(x^{2+\epsilon})\right)^2(dx^2+\hat{g}-x^2\hat{A}+O(x^{2+\epsilon})), 
$$
which implies that 
$$
H_s=-n\rho_{s,1}. 
$$
Notice that 
$v_s=\rho_s^{\frac{n-1}{2}}$ satisfies $\Delta_+v_s+\frac{n^2-1}{4}v_s=0$. Then a similar calculation as the proof of Lemma \ref{lem.asym1} gives $\rho_{s,2}$. 

\end{proof}

\subsection{Positivity}
First, we see \cite[Proposition 6.2]{CC1} still holds if the background metric has only finite regularity.  
\begin{lemma}[Case-Chang] \label{lem.J1}
If $\hat{R}> 0$ on $M$, then for all $s>\frac{n}{2}+1$, $\bar{J}_s>0$ all over $\overline{X}$. 
\end{lemma}
\begin{proof}
Let $s=\frac{n}{2}+\gamma$. If $s=n+1$, $\bar{J}_{n+1}>0$ in $X$ is proved in \cite{Le1}. 	Let $\mathcal{E} \subset (\frac{n}{2}+1,+\infty)$ be the subset such that for each $s\in \mathcal{E}$ such that $\bar{J}_{s}>0$ in $X$. It is obvious that $\mathcal{E}$ is nonempty and open.  If $s_0\in (\frac{n}{2}+1,+\infty)$ is a limit point of $\mathcal{E}$, then $\bar{J}_{s_0}\geq 0$. 
Recall $\bar{J}_s\in C^{k-1,\alpha}(X)$ and satisfies 
\begin{equation}
\frac{1}{2s-n-1}\left(\bar{\Delta}_s \bar{J}_s+(3+n-2s)\rho_s^{-1}\langle d\rho_s, d\bar{J}_s\rangle_{\bar{g}_s}\right) =-|\bar{A}_s|^2_{\bar{g}_s} + \frac{n+1}{(2s-n-1)^2} \bar{J_s}^2.
\end{equation}
By Lemma \ref{lem.asym1}, $\bar{J}_{s_0}|_M>0$. If $s_0\notin \mathcal{E}$, then there is an interior point $p\in X$ such that $\bar{J}_{s_0}(p)=0$, which contradicts with the strong maximum principle. Hence $s_0\in \mathcal{E}$ and $\mathcal{E}$ is closed. 
\end{proof}


\begin{lemma} \label{lem.T1}
If $\hat{R}>0$ on $M$, then for $s=\frac{n+1}{2}$, $\bar{T}_s>0$ and $\bar{J}_s=0$  all over $\overline{X}$.
\end{lemma}
\begin{proof}
First, we can apply the same comparison argument as the proof of \cite{GQ1} show that if $\hat{R}\geq 0$ on $M$, then $H_s>0$ on $M$. Let $v'=(\rho_{n+1})^{\frac{n-1}{2}}=(v_{n+1})^{-\frac{n-1}{2}}$. Then by Lemma \ref{lem.asym1}, 
$$
v'=x^{\frac{n-1}{2}}\left(x-\frac{n-1}{4n}\hat{J}x^3+O(x^{3+\epsilon})\right). 
$$
Moreover, by Lemma \ref{lem.asym1}, 
$$
\frac{\Delta_+v'}{v'}+\frac{n^2-1}{4}= -\frac{n-1}{2}\rho_{n+1}^2\bar{J}_{n+1}<0.
$$
Apply the maximum principle to the following equation,
$$
\Delta_+\left(\frac{v_s}{v'}\right)=-\left(\frac{\Delta_+v'}{v'}+\frac{n^2-1}{4}\right) + 2\left\langle d \left(\frac{v_s}{v'}\right), d\log v' \right\rangle_{g_+}, \quad \left.\frac{v_s}{v'}\right|_M=1
$$
we get $v_s<v'$ and hence $\rho_s<\rho_{n+1}$ in $X$. Therefore, $\bar{H}_s>0$ on M.

Second, $\bar{T}_s\in C^{k-1,\alpha}(\overline{X})$ and direct calculation shows that $\bar{T}_s$ satisfies the equation for $x=\frac{n+1}{2}$,
$$
\bar{\Delta}_s\bar{T}_s = -\frac{2}{(n-1)^2}\rho_s |\bar{E}_s|^2_{\bar{g}_s}. 
$$
By the maximum principle again, we have $\bar{T}_s>0$ in $X$. 
\end{proof}

\vspace{0.2in}
\section{Proof of Theorem \ref{thm.main}}\label{sec.6}

In this section, we prove Theorem \ref{thm.main}.
Suppose $\{(X,g^{(i)}_+)\}$ is a family of $C^{k,\alpha}$ ($k\geq 3, 0<\alpha<1$)  conformally  compact Poincar\'{e}-Einstein manifold with conformal infinity $\{(M,[\hat{g}^{(i)}])\}$ satisifying $\mathcal{Y}(M,[\hat{g}^{(i)}])\geq 0$. Let $\rho$ be a fixed smooth defining function and 
$$
\bar{g}^{(i)}=\rho^2g^{(i)}_+, \quad \hat{g}^{(i)}=\bar{g}^{(i)}|_{TM}. 
$$
For $s>\frac{n}{2}$, let $\{\bar{g}^{(i)}_s\}$ be the adapted (including Fefferman-Graham) compactification of $g^{(i)}_+$ with fixed boundary representatives: 
$$
\bar{g}_s^{(i)}|_{TM}=\bar{g}^{(i)}|_{TM}=\hat{g}^{(i)}.
$$


\subsection{Proof of Theorem \ref{thm.main}-(a)}
Assume $\{\bar{g}^{(i)}\}$ is compact in $C^{k,\alpha}(\overline{X})$ topology.

\subsubsection{Case 1: $s>\frac{n}{2}, s\neq n$} In this case
$$
\bar{g}_s^{(i)}=\left(\rho_s^{(i)}\right)^2g_+^{(i)}=\left(\phi_s^{(i)}\right)^{\frac{2}{n-s}} \bar{g}^{(i)}, \quad \phi_s^{(i)}=\left(\frac{\rho^{(i)}_s}{\rho}\right)^{n-s}, \quad \phi_s^{(i)}|_{M}=1, 
$$
and $\phi_s^{(i)}>0 $ satisfies the equation (\ref{eq.phi}) with background metric $\bar{g}^{(i)}$. 
By the global estimates given in Lemma \ref{lem.reg1}, \ref{lem.reg2} and Remark \ref{rmk.6}, there is a uniform constant $C>0$ independent of $i$ such that
\begin{equation}\label{eq.6.1}
	\|\phi_s^{(i)}\|_{C^{l,\beta}(\overline{X})} \leq C,  \quad \forall\ i,
\end{equation}
where 
\begin{itemize}
\item $l+\beta=k+\alpha$ if $k+\alpha<2s-n$ or $2s-n=2N-1$ for some positive integer $N$;
\item $l+\beta<2s-n$ if $k+\alpha\geq 2s-n$ and $2s-n$ is not a positive odd integer. 	
\end{itemize}
Therefore, choosing $0<\beta'<\beta$, we have $\{\phi_s^{(i)}\}$ is compact in $C^{l,\beta'}(\overline{X})$. 

Moreover, there exist constants $0<C_1<C_2$ independent of $i$, such that 
$$C_1<\phi_s^{(i)}<C_2,  \quad \forall\ i. $$
The uniform upper bound is obvious from (\ref{eq.6.1}). If there is no uniform lower bound, then by passing to a subsequence, we can assume 
$$
\begin{aligned}
&\bar{g}^{(i)}\longrightarrow\bar{g}^{(\infty)}\quad\textrm{in $C^{k,\alpha}(\overline{X})$}, 
\\
&\phi_{s}^{(i)}\longrightarrow \phi_{s}^{(\infty)} \quad\textrm{in $C^{l,\beta'}(\overline{X})$}, 
\end{aligned}
$$
where $\phi_{s}^{(\infty)}|_{M}=1$, $\phi_{s}^{(\infty)}\geq 0$  and $\phi_{s}^{(\infty)}(p)=0$ for some $p\in X$. However, $\phi_{s}^{(\infty)}$ satisfies the equation  (\ref{eq.phi}) with background metric $\bar{g}^{(\infty)}$. Hence $v_s^{(\infty)}=\phi_{s}^{(\infty)}\rho^{n-s}$ satisfies the equation (\ref{eq.def1}) with background metric $g_+^{(\infty)}=\rho^{-2}\bar{g}^{(\infty)}$. Since  $Y(M,[\hat{g}^{(\infty)}])\geq 0$ and  $\mathrm{Spec(-\Delta_+^{(\infty)})}>\frac{n^2}{4}$, as a solution $v^{(\infty)}_s(p)=0$  can not happen. 

Finally, the compactness of  $\{\phi_s^{(i)}\}$ in $C^{l,\beta'}(\overline{X})$ and uniform boundness $0<C_1<\phi_s^{(i)}<C_2$ imply that $\{\bar{g}_s^{(i)}\}$ is compact in $C^{l,\beta'}(\overline{X})$ topology.

\subsubsection{Case 2: $s=n$}. In this case, 
$$
\bar{g}^{(i)}_n=\left(\rho^{(i)}_n\right)^2g_+=e^{2\varphi^{(i)}}\bar{g}^{(i)}, \quad \varphi^{(i)}=w^{(i)}-\log\rho, \quad \varphi ^{(i)}|_M=0, 
$$
and $w^{(i)}$ satisfies the equation (\ref{eq.def2}), $\varphi ^{(i)}$ satisfies the equation (\ref{eq.vphi}) with background metric $\bar{g}^{(i)}$.
By the global estimates given in  Lemma \ref{lem.reg3}, \ref{lem.reg4} and Remark \ref{rmk.7}, there is a uniform constant $C>0$ independent of $i$ such that
\begin{equation}\label{eq.6.2}
	\|\varphi^{(i)}\|_{C^{l,\beta}(\overline{X})} \leq C, \quad \forall\ i, 
\end{equation}
where 
\begin{itemize}
\item $l+\beta=k+\alpha$ if $k+\alpha<n$ or $n$ is odd;
\item $l+\beta<n$ if $k+\alpha\geq n$ and $n$ is even. 
\end{itemize}
Therefore, choosing $0<\beta'<\beta$, we have $\{\phi_s^{(i)}\}$ is compact in $C^{l,\beta'}(\overline{X})$. And (\ref{eq.6.2}) implies that there is a uniform constant $C_1>0$ such that 
$$
-C_1<\varphi^{(i)}<C_1, \quad \forall\ i. 
$$
These imply that $\{\bar{g}_n^{(i)}\}$ is compact in $C^{l,\beta'}(\overline{X})$ topology.

\subsection{Proof of Theorem \ref{thm.main}-(b)}
Assume   $\{\bar{g}_s^{(i)}\}$ is compact in $C^{k,\alpha}(\overline{X})$ topology  for some $s>\frac{n}{2}+1$ or $s=\frac{n+1}{2}$ and the boundary scalar curvature $\hat{R}^{(i)}>0$ for all $i$.   We first derive some uniform  estimates for $\rho_s$. 

\subsubsection{Case 1:  $s>\frac{n}{2}+1$}
In this case, first notice that
$\rho^{(i)}_s$ satisfies the following equation for each $i$ with background metric $\bar{g}^{(i)}$:
\begin{equation}\label{eq.S1}
\begin{cases}
\displaystyle \bar{\Delta}_s\rho_s + \frac{2s}{2s-n-1}\bar{J}_s\rho_s=0, &\textrm{in $X$}, 
\\
\displaystyle \rho_s=0, &\textrm{on $M$}. 
\end{cases}
\end{equation}
where 
$$
\bar{J}_s =\frac{2s-n-1}{2}\left(\frac{1-|d\rho_s|^2_{\bar{g}_s}}{\rho_s^2}\right)\in C^{k-2,\alpha}(\overline{X}).
$$
Recall that
$$
T_s=\frac{1-|d\rho_s|^2_{\bar{g}_s}}{\rho_s}. 
$$
Notice that (\ref{eq.S1}) is also equivalent to 
\begin{equation}\label{eq.S3}
\begin{cases}
\displaystyle \bar{\Delta}_s\rho_s =-s\bar{T}_s, &\textrm{in $X$}, 
\\
\displaystyle \rho_s=0, &\textrm{on $M$}. 
\end{cases}
\end{equation}
And $\bar{T}_s$ satisfies
\begin{equation}\label{eq.S4}
\begin{cases}
\displaystyle 
 \bar{\Delta}_s \bar{T}_s 
 +\bar{J}_s\bar{T}_s 
 =-2|\bar{A}_s|_{\bar{g}_s}^2\rho_s+2\langle d\rho_s, d\bar{J}_s\rangle_{\bar{g}_s},  &\textrm{in $X$}, 
 \\
\displaystyle  \bar{T}_s=0, &\textrm{on $M$}. 
 \end{cases}
\end{equation}
where $\bar{A}_s$ is the Schouten tensor of $\bar{g}_s$. 

\begin{lemma}\label{lem.S1}
For $s>\frac{n}{2}+1$, if  $\{\bar{g}_s^{(i)}\}$ is compact in $C^{k,\alpha}(\overline{X})$ topology and $\hat{R}^{(i)}>0$ for all $i$, then there exist constants $C>c>0$ such that for all $i$, 
\begin{itemize}
\item[(i)]
$\displaystyle \|\rho_s^{(i)}\|_{L^{\infty}(X)}\leq C$;
\item[(ii)] 
$\displaystyle \|\rho_s^{(i)}\|_{C^{k+1,\alpha}(\overline{X})}\leq C$;
\item[(iii)]
$\displaystyle  c\rho<\rho_s^{(i)}<C\rho$; 
\item[(iv)] 
$\displaystyle  c<|d\rho^{(i)}_s|_{\bar{g}^{(i)}_s}\leq 1$ if $ \rho<c$.
\end{itemize}
\end{lemma}
\begin{proof}
For (i), notice that $\hat{R}^{(i)}>0$ implies that $\bar{J}^{(i)}_s>0$  by Lemma \ref{lem.J1}. Therefore,
$$
|d\rho^{(i)}_s|_{\bar{g}^{(i)}_s}\leq 1
\quad\Longrightarrow \quad 
\|\rho_s^{(i)}\|_{L^{\infty}(X)}\leq \mathrm{diam}(\overline{X},\bar{g}_s^{(i)}), 
\quad \forall \ i.
$$

For (ii), since $\{\bar{g}_s^{(i)}\}$ is $C^{k,\alpha}(\overline{X})$ compact, there is a uniform constant $C_1>0$ such that
$$
\|\bar{A}^{(i)}_s\|_{C^{k-2,\alpha}(\overline{X})}\leq C_1, \quad
\|\bar{J}^{(i)}_s\|_{C^{k-2,\alpha}(\overline{X})}\leq C_1, \quad \forall\  i.
$$
By the classical Schauder estimates for equation (\ref{eq.S1}), we first get
$$
\|\rho_s^{(i)}\|_{C^{k,\alpha}(\overline{X})}\leq C_2, \quad \forall\  i.
$$
for some constant $C_2>0$. Then (\ref{eq.S4}) implies that
$$
 \|\bar{T}_s^{(i)}\|_{C^{k-1,\alpha}(\overline{X})}\leq C_3, \quad \forall\  i.
$$
for some constant $C_3>0$. And hence (\ref{eq.S3}) implies that
$$
\|\rho_s^{(i)}\|_{C^{k+1,\alpha}(\overline{X})}\leq C_4, \quad \forall\  i.
$$
for some constant $C_4>0$. 

For (iii), since $\rho$ is a fixed coordinates in a collar neighborhood $X_c$, the uniform upper bound of $\rho_s^{(i)}/\rho$ is directly from (ii). If there is no uniform lower bound, then by passing to a subsequence, we can assume 
$$
\begin{aligned}
&\rho_s^{(i)}\longrightarrow \rho_s^{(\infty)}, \quad\textrm{in $C^{k,\alpha'}(\overline{X})$}, 
\\
&\bar{g}_s^{(i)}\longrightarrow  \bar{g}_s^{(\infty)}, \quad\textrm{in $C^{k,\alpha}(\overline{X})$}.
\end{aligned}
$$
and $(\rho_s^{(\infty)}/\rho)|_{M}=1$, $(\rho_s^{(\infty)}/\rho)|_{p}=0$ for some $p\in X$. However,  $\rho_s^{(\infty)}$ satisfies the equation (\ref{eq.S1}) with background metric $\bar{g}_s^{(\infty)}$ and $\bar{J}_s^{(\infty)}>0$ in $X$. By the strong maximum principle, $\rho_s^{(\infty)}>0$ in $X$, which contradicts with $\rho_s^{(\infty)}(p)=0$. Hence there must be some constants $0<c<C$ such that 
$$
c<\frac{\rho_s^{(i)}}{\rho}<C, \quad \forall \ i. 
$$

For (iv), notice that $|d\rho^{(i)}_s|^2_{\bar{g}^{(i)}_s}=1-\frac{2}{2s-n-1}\bar{J}^{(i)}_s (\rho^{(i)}_s)^2$. So it follows immediately from (iii). 
\end{proof}

\subsubsection{Case 2:  $s=\frac{n+1}{2}$} In this case, notice that $\rho_s^{(i)}$ satisfies the following equation 
\begin{equation}\label{eq.S2}
\begin{cases}
\displaystyle \bar{\Delta}_s\rho_s =-s \bar{T}_s, & \textrm{in $X$},
\\
\displaystyle \rho_s =0, & \textrm{on $M$},
\end{cases}
\end{equation}
where
 $$\bar{T}_s =\frac{1-|d\rho_s|^2_{\bar{g}_s}}{\rho_s}$$ 
satisfies 
\begin{equation}\label{eq.T2}
\begin{cases}
\displaystyle \bar{\Delta}_s\bar{T}_s = -\frac{2}{(n-1)^2} |\bar{E}_s|^2_{\bar{g}_s}\rho_s , & \textrm{in $X$}, 
\\
\displaystyle \bar{T}_s = \frac{2}{n}H_s, & \textrm{on $M$}. 
\end{cases}
\end{equation}
Here $\bar{E}_s$ is the trace free Ricci tensor for $\bar{g}_s$.

\begin{lemma}\label{lem.S2}
For $s=\frac{n+1}{2}$, if  $\{\bar{g}_s^{(i)}\}$ is compact in $C^{k,\alpha}(\overline{X})$ topology and $\hat{R}^{(i)}>0$ for all $i$, then there exist constants $C>c>0$ such that for all $i$, 
\begin{itemize}
\item[(i)]
$\displaystyle \|\rho_s^{(i)}\|_{L^{\infty}(X)}\leq C$ and $\displaystyle \|\bar{T}_s^{(i)}\|_{L^{\infty}(X)}\leq C$;
\item[(ii)] 
$\displaystyle \|\rho_s^{(i)}\|_{C^{k+1,\alpha}(\overline{X})}\leq C$ and $\displaystyle \|\bar{T}_s^{(i)}\|_{C^{k-1,\alpha}(\overline{X})}\leq C$; 
\item[(iii)]
$\displaystyle  c\rho<\rho_s^{(i)}<C\rho$; 
\item[(iv)] 
$\displaystyle  c<|d\rho^{(i)}_s|_{\bar{g}^{(i)}_s}\leq 1$ if $ \rho<c$.
\end{itemize}
\end{lemma}
\begin{proof}
For (i), first by Lemma \ref{lem.T1}, $\bar{T}_s>0$ implies that 
$$
|d\rho^{(i)}_s|_{\bar{g}^{(i)}_s}\leq 1
\quad\Longrightarrow \quad
 \|\rho_s^{(i)}\|_{L^{\infty}(X)}\leq \mathrm{diam}(\overline{X}, \bar{g}_s^{(i)}),
\quad \forall \ i.
$$
If there is no uniform constant $C$ such that $\displaystyle \|\bar{T}_s^{(i)}\|_{L^{\infty}(X)}\leq C$, then by passing to a subsequence,  we can assume 
$$
\begin{aligned}
& \|\bar{T}_s^{(i)}\|_{L^{\infty}(X)}\longrightarrow +\infty. 
\end{aligned}
$$
Take $\tilde{\rho}^{(i)}=\rho_s^{(i)}/\|\bar{T}_s^{(i)}\|_{L^{\infty}(X)}$ and $\tilde{T}^{(i)}=\bar{T}_s^{(i)}/\|\bar{T}_s^{(i)}\|_{L^{\infty}(X)}$. Then  $(\tilde{\rho}^{(i)}, \tilde{T}^{(i)})$ satisfies
$$
\begin{cases}
\displaystyle \bar{\Delta}^{(i)}_s \tilde{\rho}^{(i)} +s \tilde{T}^{(i)}=0, & \textrm{in $X$},
\\
\displaystyle  \tilde{\rho}^{(i)}  =0, & \textrm{on $M$},
\end{cases}
\quad\mathrm{and}\quad
\begin{cases}
\displaystyle \bar{\Delta}^{(i)}_s\tilde{T}^{(i)} +\frac{2}{(n-1)^2}\tilde{\rho}^{(i)}   |\bar{E}^{(i)}_s|^2_{\bar{g}_s}=0, & \textrm{in $X$}, 
\\
\displaystyle \tilde{T}^{(i)} = \frac{2}{n} \frac{H_s}{\|\bar{T}_s^{(i)}\|_{L^{\infty}(X)}}, & \textrm{on $M$}. 
\end{cases}
$$
Notice that
$\{\bar{g}_s^{(i)}\}$ is compact in $C^{k,\alpha}(\overline{X})$ topology implies that  
$$
\|H_s^{(i)}\|_{C^{k-1,\alpha}(M)} \leq C_1, \quad
\|\bar{E}^{(i)}_s\|_{C^{k-2,\alpha}(\overline{X})}\leq C_1, 
$$ 
for some constant $C_1>0$ independent of $i$. 
Hence applying the classical Schauder estimate to above coupled system, we have 
$$
 \|(\tilde{\rho}^{(i)}, \tilde{T}^{(i)})\|_{C^{k-1,\alpha}(\overline{X})} \leq C',
$$
for some constant $C'>0$  independent of $i$. 
By passing to a subsequence again, we assume 
$$
\begin{aligned}
& \tilde{T}^{(i)}\longrightarrow \tilde{T}^{(\infty)},  \quad \textrm{in $C^2(\overline{X})$}, 
\\
& \bar{g}_s^{(i)}\longrightarrow  \bar{g}_s^{(\infty)}, \quad\textrm{in $C^{k,\alpha}(\overline{X})$}. 
\end{aligned}
$$
Then $ \tilde{T}^{(\infty)}$ satisfies 
$$
\begin{cases}
\bar{\Delta}_{s}^{(\infty)}\tilde{T}^{(\infty)}=0, & \quad \textrm{in $X$}, 
\\
\tilde{T}^{(\infty)}= 0, & \quad \textrm{on $M$}.
\end{cases}
$$
This implies that $\tilde{T}^{(\infty)}\equiv 0$ which contradicts with $\|\tilde{T}^{(\infty)}\|_{L^{\infty}(X)}=1$.

For (ii),  it follows from (i) together with the classical Schauder estimates for the coupled equations (\ref{eq.S2}) and (\ref{eq.T2}). 

For (iii),  the uniform upper bound of $\rho_s^{(i)}/\rho$ is directly from (ii). If there is no uniform lower bound, then by passing to a subsequence, we can assume 
$$
\begin{aligned}
&\rho_s^{(i)}\longrightarrow \rho_s^{(\infty)}, \quad\textrm{in $C^{k,\alpha'}(\overline{X})$}, 
\\
&\bar{g}_s^{(i)}\longrightarrow  \bar{g}_s^{(\infty)}, \quad\textrm{in $C^{k,\alpha}(\overline{X})$}, 
\end{aligned}
$$
and $(\rho_s^{(\infty)}/\rho)|_{M}=1$, $(\rho_s^{(\infty)}/\rho)|_p=0$ for some $p\in X$. 
Here  $\rho_s^{(\infty)}$ and $\bar{T}_s^{(\infty)}$ satisfies the equation (\ref{eq.S2}) and (\ref{eq.T2}) with background metric $\bar{g}^{(\infty)}_s$.  And $H_s^{(\infty)}\geq  0$ on  $M$ implies that $\bar{T}_s^{\infty}>0$ in $X$. By the strong maximum principle, $\rho_s^{(\infty)}>0$ in $X$, which contradicts with $\rho_s^{(\infty)}(p)=0$. Hence there must be some uniform constants $0<c<C$ such that 
$$
c<\frac{\rho_s^{(i)}}{\rho}<C, \quad \forall \ i. 
$$

For (iv),  notice that $|d\rho^{(i)}_s|^2_{\bar{g}^{(i)}_s} =1-\rho_s^{(i)}\bar{T}_s^{(i)}$. Hence (iv) follows immediately from (iii). 
\end{proof}

Now we continue the proof of Theorem \ref{thm.main}-(b).
Lemma \ref{lem.S1} and  \ref{lem.S2} imply that for $s>\frac{n}{2}+1$ or $s=\frac{n+1}{2}$, 
there is a  constant $C'>0$ such that
$$
\|\rho_s^{(i)}\|_{C^{k+1,\alpha}(\overline{X})}\leq C, \quad \forall\  i.
$$
Since $\rho_s|_M=0$, by Corollary \ref{cor.ext} in the Appendix, there is a constant $C'>0$ such that
$$
\left\| \frac{\rho_s^{(i)}}{\rho}\right\|_{C^{k,\alpha}(\overline{X})} \leq C', \quad \forall\  i.
$$
This together with    Lemma \ref{lem.S1}-(iii) and Lemma  \ref{lem.S2}-(iii)  implies that for some constant $C''>0$, 
$$
\left\| \frac{\rho}{\rho_s^{(i)}}\right\|_{C^{k,\alpha}(\overline{X})} \leq C'', \quad \forall\  i. 
$$
Recall  that
$$
\bar{g}^{(i)}=\left( \frac{\rho}{\rho_s^{(i)}}\right)^2\bar{g}_s^{(i)}. 
$$ 
Therefore $\{\bar{g}^{(i)}\}$ 
is compact in $C^{k,\alpha'}(\overline{X})$ topology for $0<\alpha'<\alpha$. We finish the proof.

\vspace{0.2in}
\section{Appendix}\label{sec.app}

In this section we derive an extension theorem. 

\begin{lemma}\label{lem.ext}
For any $k,l\geq 0$ and $0\leq \alpha<1$, suppose  $f=f(y)$ is a $C^{k,\alpha}$ function which has compact support in $\mathbb{R}^n_y$. 
Then there exists an extension $F=F(x,y)$ satisfying the following properties:
\begin{itemize}
\item[(1)]
$F\in C^{k+l,\alpha}([0,1)_x\times \mathbb{R}^n_y)$ and there exists some  constant $C_1>0$ only depending on $(k,l,\alpha)$ such that
$$
\|F\|_{C^{k+l,\alpha}([0,1)_x\times \mathbb{R}^n_y)} \leq C_1\|f\|_{C^{k,\alpha}(\mathbb{R}^n_y)}; 
$$
\item[(2)]
As $x\rightarrow 0$, $F(x,y)$ has an asymptotical expansion
$$
F(x,y)=x^lf(y)+x^{l+1}f_1(y)+\cdots x^{l+k}f_k(y)+O(x^{l+k+\alpha})
$$
where $f_i(y)\in C^{k-i,\alpha}(\mathbb{R}_y)$ determined  derivatives of $f$ up to $i$-th order and there exists a constant $C_2>0$ only depending on $(k,l,\alpha)$, 
$$
\|f_i\|_{C^{k-i,\alpha}(\mathbb{R}_y)}\leq C_2\|f\|_{C^{k,\alpha}(\mathbb{R}^n_y)}, \quad
\forall\  0\leq i\leq k;
$$

\item[(3)]
For any $j\geq 0$ and multi-index $J\geq 0$, there exists a constant $C_3>0$ only depending on $(k,l,\alpha,j,J)$ such that
$$
\|x^{j+|J|}\partial_x^j\partial_y^{J}F\|_{C^{k+l,\alpha}([0,1)_x\times \mathbb{R}^n_y)}  \leq   C_3 \|f\|_{C^{k,\alpha}(\mathbb{R}^n_y)}.
$$

\end{itemize}
\end{lemma}

\begin{proof}
Let $\phi\in C_c^{\infty}(\mathbb{R}^n_y)$ satisfy
$$
\begin{gathered}
0\leq \phi\leq 1, \quad \phi(0)=1, \quad \phi(y)=0\textrm{ if $|y|\geq 1$}
\quad
\mathrm{and}\quad\int_{\mathbb{R}^n_y}\phi(y)dy=1. 
\end{gathered}
$$
For any $0<x<1$, denote $\displaystyle \phi_{x}(y)=\frac{1}{x^{n}} \phi\left(\frac{y}{x}\right)$ and
$$
f_x(y)=\phi_x * f (y) =\int_{\mathbb{R}^n}\phi_x(y')f(y-y')dy' = \int_{\mathbb{R}^n} \phi(z) f(y-xz)dz. 
$$
Then it is easy to check that $f_{x}\in C^{\infty}_c(\mathbb{R}^n_y)$ and 
$$
\lim_{x\rightarrow 0} \|f_x-f\|_{C^{k,\alpha}(\mathbb{R}^n_y)} =0. 
$$
Hence there is some constant $C_0>0$ such that
$$
\|f_x\|_{C^{k,\alpha}(\mathbb{R}^n_y)}\leq C_0+\|f\|_{C^{k,\alpha}(\mathbb{R}^n_y)}, \quad \forall\  0<x<1.
$$
Denote $h(x,y)=f_{x}(y)$. Then $h$ satisfies the following properties: 
\begin{itemize}
\item[(i)]
If $I$ is a multi index such that $|I|\leq k$, then
$$
\partial_y^{I}h(x,y) =\phi_x  * \left(\partial_y^{I}f\right)=\int_{\mathbb{R}^n} \phi(z) f^{(I)}(y-xz)dz, 
$$
and since $\mathrm{supp}(\phi)\subset \{|z|<1\}$, 
$$
\begin{aligned}
&\ \frac{|\partial_y^{I}h(x,y) -\partial_y^{I}h(x',y')|}{|(x,y)-(x',y')|^{\alpha}}
\\
\leq &\  \int_{\mathbb{R}^n} \phi(z)  \frac{|f^{(I)}(y-xz)-f^{(I)}(y'-x'z)|}{|(x,y)-(x',y')|^{\alpha}} dz
\\
\leq &  \int_{\mathbb{R}^n} \phi(z) \left( \frac{|f^{(I)}(y-xz)-f^{(I)}(y'-xz)|}{|y-y'|^{\alpha}} + \frac{|f^{(I)}(y'-xz)-f^{(I)}(y'-x'z)|}{|x-x'|^{\alpha}} \right)dz.
\end{aligned}
$$
These imply that for some constant $C>0$ only depending on $\phi$,
$$
\| \partial_y^{I}h \|_{C^{0,\alpha}((0,1)\times \mathbb{R}^n_y)} \leq C \|f\|_{C^{k,\alpha}(\mathbb{R}^n_y)}.
$$

\item[(ii)]
If $I$ is a multi index such that $|I|> k$, then choosing a multi index $J$ with $J\leq I$ and $|J|=k$, we have
$$
\partial_y^{I}h =\left(\partial_y^{I-J}\phi_x\right) * \left(\partial_y^{J}f\right)
= \frac{1}{x^{|I|-k}} \int_{\mathbb{R}^n} \phi^{(I-J)}(z) f^{(J)}(y-xz)dz. 
$$
Hence similar as (i), for some constant $C>0$ only depending on $k,|I|$ and $\phi$, we have
$$
\|x^{|I|-k}\partial_y^{I} h\|_{C^{0,\alpha}((0,1)\times \mathbb{R}^n_y)} \leq C \|f\|_{C^{k,\alpha}(\mathbb{R}^n_y)}. 
$$

\item[(iii)]
If $i\leq k$, then 
$$
\partial_x^i h = \sum_{|I|=i}\int_{\mathbb{R}^n} \phi(z) (-z)^{I} \partial_y^{I}f(y-xz)dz
$$
Hence similar as (i),  for some constant  constant $C>0$ only depending on $\phi$ we have
$$
\|\partial_x^i f_x(y) \|_{C^{0,\alpha}((0,1)\times\mathbb{R}^n_y)} \leq C  \|f\|_{C^{k,\alpha}(\mathbb{R}^n_y)}.
$$

\item[(iv)]
If $i>k$, then 
$$
\begin{aligned}
\partial_x^i h
=&\  \partial_x^{i-k} \left(\sum_{|J|=k} \int_{\mathbb{R}^n} \phi(z) (-z)^{J} f^{(J)}(y-xz)dz\right)
\\
=&\ \partial_x^{i-k} \left(\sum_{|J|=k} \int_{\mathbb{R}^n} \frac{1}{x^{n+k}}\phi \left(\frac{y'}{x}\right) (-y')^{J} f^{(J)}(y-y')dy'\right)
\\
=&\ \sum_{|I|\leq i-k} \sum_{|J|=k}  \int_{\mathbb{R}^n}  \frac{c_{I}}{x^{n+i+|J|}} \phi^{(I)}\left(\frac{y'}{x}\right)  (y')^{I+J} f^{(J)}(y-y')dy'
\\
=&\ \sum_{|I|\leq i-k} \sum_{|J|=k}  \frac{c_{I}}{x^{i-k}}   \int_{\mathbb{R}^n} \phi^{(I)}(z)  z^{I+J} f^{(J)}(y-xz)dz. 
\end{aligned}
$$
Here $c_{I}$ are constants. For some constant $C>0$ only depending on $i,k$ and $\phi$, we have 
$$
\|x^{i-k}\partial_x^i f_x(y) \|_{C^{0,\alpha}((0,1)\times \mathbb{R}^n_y)} \leq C \|f\|_{C^{k,\alpha}(\mathbb{R}^n_y)}. 
$$
\end{itemize}

Now we define
$$
F(x,y)=\chi(x)x^l h(x,y),\quad  0<x<1.
$$
where $\chi(x)$ is a smooth cutoff function satisfying 
$$
\textrm{
$\chi\in C_c^{\infty}([0,1))$, $0\leq \chi\leq 1$, 
$ \chi(x)=1$ if $0\leq x<\frac{1}{2}$ and $\chi(x)=0$ if $x>\frac{3}{4}$. }
$$
And $F(0,y)=f(y)$ if $l=0$; $F(0,y)=0$ if $l>0$. 
By (i)-(iv), there exists some constant $C>0$ only depending on $k,l$ and $\phi$ such that
$$
\|F\|_{C^{k+l,\alpha}([0,1)_x\times\mathbb{R}^n_y} \leq C \|f\|_{C^{k,\alpha}(\mathbb{R}^n_y)}. 
$$
Hence  $F$ satisfies (1).  

For the asymptotical expansion of $F$,  it is simply by taking $x\rightarrow 0$ in (iii). 
For $i\leq k$, 
$$
\lim_{x\rightarrow 0} \partial_x^i h = \sum_{|I|=i}\left(\int_{\mathbb{R}^n} \phi(z) (-z)^{I} dz\right)f^{(I)}(y)
$$
In particular, there is some constant $C>0$ depending on $k,\alpha$ such that
$$
\begin{aligned}	
 \left|\frac{\partial_x^k h(x,y) - \partial_x^kh(0,y)}{x^{\alpha}}\right|
 =&\ \left|\sum_{|I|=k}\int_{\mathbb{R}^n} \phi(z) (-z)^{I} 
 \left(\frac{f^{(I)}(y-xz)-f^{(I)}(y)}{x^{\alpha}} \right) dz\right|
 \\
 =&\ \sum_{|I|=k}\int_{\mathbb{R}^n} \phi(z) |z|^{|I|+\alpha}
 \left|\frac{f^{(I)}(y-xz)-f^{(I)}(y)}{|xz|^{\alpha}} \right| dz
 \\
 \leq & \ C\|f\|_{C^{k,\alpha}(\mathbb{R}^n_y)}. 
\end{aligned}
$$
Hence $F$ satisfies (2). 

To prove (3), we just notice that
$$
\begin{aligned}
x^{j+J}\partial_x^j \partial_y^J h
=&\ x^j\partial_x^j \int_{\mathbb{R^n}} \frac{1}{x^n}\phi^{(J)} \left( \frac{y'}{x}\right) f(y-y') dy'
\\
=&\ \sum_{|I|\leq j} \int_{\mathbb{R}^n} \frac{c_I}{x^{n+|I|}} \phi^{(I+J)} \left( \frac{y'}{x}\right) (y')^I f(y-y') dy'
\\
=&\ \sum_{|I|\leq j} \int_{\mathbb{R}^n} c_I \phi^{(I+J)} (z) z^I f(y-xz) dz, 
\end{aligned}
$$
where $c_I$ are some constants. 
Replace $\phi(z)$ by $\tilde{\phi}(z)=\sum_{|I|\leq i}c_I \phi^{(I+J)} (z) z^I$ and $h$ by  $x^{j+J}\partial_x^j \partial_y^J h$, then (i)-(iv) are also valid with an adjustment of the constants. Then (3) follows from a similar proof of (1). 
\end{proof}

\begin{theorem}\label{thm.ext}
Suppose $\overline{X}$ is a smooth compact manifold with boundary $M$. Let $\rho$ be a smooth boundary defining function. For any integers $k,l\geq 0$ and $0\leq \alpha<1$, if $f\in C^{k,\alpha}(M)$, then there exists $F\in C^{k+l,\alpha}(\overline{X})\cap \Lambda_{l}^{\infty}(X)$ satisfies the following properties:
\begin{itemize}
\item[(1)]
There exists some constant $C_1>0$ only depending on $k,l,\alpha$, such that
$$
\|F\|_{C^{k+l,\alpha}(\overline{X})} \leq C_1 \|f\|_{C^{k,\alpha}(M)}; 
$$	
\item[(2)]
As $\rho\rightarrow 0$, $F$ has an asymptotical expansion as follows
$$
F=\rho^lf_0+\rho^{l+1}f_1+\cdots +\rho^{l+k} f_k +O(x^{k+l+\alpha})
$$
where $f_i\in C^{k-i,\alpha}(M)$ with $f_0=f$ and there exists some constant $C_2>0$ only depending on $k,l,\alpha$, 
$$
\|f_i\|_{C^{k-i,\alpha}(M)} \leq C_2\|f\|_{C^{k,\alpha}(M)}, \quad \forall \ i=1,...,k.
$$
\item[(3)]
If $P$ is a differential operators of order $m\leq k+l$ with coefficients in $C^{k,\alpha}(\overline{X})$, then
there exists some constant $C_3>0$ only depending on the  $C^{k,\alpha}(\overline{X})$ norm of the coefficients such that
$$
\|PF\|_{C^{k+l-m,\alpha}(\overline{X})} \leq C_3 \|f\|_{C^{k,\alpha}(M)}, 
\quad 
\|\rho^mPF\|_{C^{k+l,\alpha}(\overline{X})} \leq C_3 \|f\|_{C^{k,\alpha}(M)}. 
$$
\end{itemize}
\end{theorem}

\vspace{0.1in}
Theorem \ref{thm.ext} directly implies the following.

\begin{corollary}\label{cor.ext}
Suppose $\overline{X}$ is a smooth compact manifold with boundary $M$. Let $\rho$ be a smooth boundary defining function. For any integer $k\geq 0$ and $0\leq \alpha<1$, if $f\in C^{k+1,\alpha}(\overline{X})$ satisfying $f|_{M}=0$, then $f/\rho\in C^{k,\alpha}(\overline{X})$. Moreover, there exists a constant $C>0$ independent of $f$ such that
$$
\left\|\frac{f}{\rho}\right\|_{C^{k,\alpha}(\overline{X})} \leq C \|f\|_{C^{k+1,\alpha}(\overline{X})}. 
$$
\end{corollary}
\begin{proof}
	First, $f\in C^{k+1,\alpha}(\overline{X})$ and $f|_{M}=0$ implies that $f$ has Taylor expansion as the following: 
$$
f=\rho f_1+O(\rho^{2}), 
$$
where $ f_1\in C^{k,\alpha}(M)$ and for some constant $C_1>0$, 
$$
 \|f_1\|_{C^{k,\alpha}(M)}\leq  C_1\|f\|_{C^{k+1,\alpha}(\overline{X})}. 
$$
Take $\tilde{f}_1$ be the extension of $\rho f_1$ as in Theorem \ref{thm.ext}. Then for some constant $C_1'>0$, 
$$
 \|\tilde{f}_1\|_{C^{k+1,\alpha}(\overline{X})}\leq C_1' \|f_1\|_{C^{k,\alpha}(M)}. 
$$
Moreover, 
$$
f-\tilde{f}_1=\rho^2f_2+O(\rho^{3}) 
$$
where $ f_2\in C^{k-1,\alpha}(M)$ and for some constant $C_2>0$, 
$$
 \|f_2\|_{C^{k-1,\alpha}(M)}\leq  C_2\|f\|_{C^{k+1,\alpha}(\overline{X})}. 
$$
Repeat above progress until we get
$$
f=\tilde{f}_1+\tilde{f}_2+\cdots +\tilde{f}_k + E, 
$$
where $\tilde{f}_i\sim \rho^if_i$  and  for some constants $C_i, C_i'>0$
$$
  \|f_i\|_{C^{k+1-i,\alpha}(M)}\leq  C_i\|f\|_{C^{k+1,\alpha}(\overline{X})}, 
  \quad
   \|\tilde{f}_i\|_{C^{k+1,\alpha}(\overline{X})}\leq C_1' \|f_i\|_{C^{k+1-i,\alpha}(M)}, 
   \quad i=1,...,k. 
$$
Hence for some constant $C>0$, 
$$
 \|E\|_{C^{k+1,\alpha}(\overline{X})}\leq C\|f\|_{C^{k+1,\alpha}(\overline{X})}, 
 \quad E=O(\rho^{k+1+\alpha}).
$$
Therefore , 
$$
E\in C^{k+1,\alpha}_{(k+1+\alpha)}(\overline{X}) \hookrightarrow \Lambda^{k+1,\alpha}_{k+1+\alpha}(X),
\quad\Longrightarrow\quad
\frac{E}{\rho} \in \Lambda^{k+1,\alpha}_{k+\alpha}(X) \hookrightarrow C^{k,\alpha}(\overline{X}). 
$$
By Lemma \ref{lem.FS3}-\ref{lem.FS2},  $C^{k+1,\alpha}_{(k+1+\alpha)}(\overline{X}) $ is closed in $C^{k+1,\alpha}(\overline{X}) $ and above two inclusion maps are continuous.
Hence for some constant $C'>0$, 
$$
\left\|\frac{E}{\rho }\right\|_{C^{k,\alpha}(\overline{X})} \leq C' \|E\|_{C^{k+1,\alpha}(\overline{X})}.  
$$
From the extension method given in Lemma \ref{lem.ext} and Theorem \ref{thm.ext}, it is easy to see
for some constant $C''>0$, 
$$
\left\|\frac{\tilde{f}_i}{\rho}\right\|_{C^{k,\alpha}(\overline{X})} \leq C'' \|\tilde{f}_i\|_{C^{k+1,\alpha}(\overline{X})}, \quad i=1,2,...,k.
$$
We finish the proof. 
\end{proof}


\end{document}